\documentclass[twoside,leqno,10pt, A4]{amsart}
\usepackage{amsfonts}
\usepackage{amsmath}
\usepackage{amscd}
\usepackage{amssymb}
\usepackage{amsthm}
\usepackage{amsrefs}
\usepackage{latexsym}
\usepackage{mathrsfs}
\usepackage{bbm}
\usepackage{amscd}
\usepackage{amssymb}
\usepackage{amsthm}
\usepackage{amsrefs}
\usepackage{latexsym}
\usepackage{mathrsfs}
\usepackage{bbm}
\usepackage{enumerate}
\usepackage{graphicx}
\usepackage{color}
\begin{document}

\newtheorem{theorem}[subsection]{Theorem}
\newtheorem{proposition}[subsection]{Proposition}
\newtheorem{lemma}[subsection]{Lemma}
\newtheorem{corollary}[subsection]{Corollary}
\newtheorem{conjecture}[subsection]{Conjecture}
\newtheorem{prop}[subsection]{Proposition}
\newtheorem{defin}[subsection]{Definition}

\numberwithin{equation}{section}
\newcommand{\mr}{\ensuremath{\mathbb R}}
\newcommand{\mc}{\ensuremath{\mathbb C}}
\newcommand{\dif}{\mathrm{d}}
\newcommand{\intz}{\mathbb{Z}}
\newcommand{\ratq}{\mathbb{Q}}
\newcommand{\natn}{\mathbb{N}}
\newcommand{\comc}{\mathbb{C}}
\newcommand{\rear}{\mathbb{R}}
\newcommand{\prip}{\mathbb{P}}
\newcommand{\uph}{\mathbb{H}}
\newcommand{\fief}{\mathbb{F}}
\newcommand{\majorarc}{\mathfrak{M}}
\newcommand{\minorarc}{\mathfrak{m}}
\newcommand{\sings}{\mathfrak{S}}
\newcommand{\fA}{\ensuremath{\mathfrak A}}
\newcommand{\mn}{\ensuremath{\mathbb N}}
\newcommand{\mq}{\ensuremath{\mathbb Q}}
\newcommand{\half}{\tfrac{1}{2}}
\newcommand{\f}{f\times \chi}
\newcommand{\summ}{\mathop{{\sum}^{\star}}}
\newcommand{\chiq}{\chi \bmod q}
\newcommand{\chidb}{\chi \bmod db}
\newcommand{\chid}{\chi \bmod d}
\newcommand{\sym}{\text{sym}^2}
\newcommand{\hhalf}{\tfrac{1}{2}}
\newcommand{\sumstar}{\sideset{}{^*}\sum}
\newcommand{\sumprime}{\sideset{}{'}\sum}
\newcommand{\sumprimeprime}{\sideset{}{''}\sum}
\newcommand{\sumflat}{\sideset{}{^\flat}\sum}
\newcommand{\shortmod}{\ensuremath{\negthickspace \negthickspace \negthickspace \pmod}}
\newcommand{\V}{V\left(\frac{nm}{q^2}\right)}
\newcommand{\sumi}{\mathop{{\sum}^{\dagger}}}
\newcommand{\mz}{\ensuremath{\mathbb Z}}
\newcommand{\leg}[2]{\left(\frac{#1}{#2}\right)}
\newcommand{\muK}{\mu_{\omega}}
\newcommand{\thalf}{\tfrac12}
\newcommand{\lp}{\left(}
\newcommand{\rp}{\right)}
\newcommand{\Lam}{\Lambda_{[i]}}
\newcommand{\lam}{\lambda}
\newcommand{\af}{\mathfrak{a}}
\newcommand{\sw}{S_{[i]}(X,Y;\Phi,\Psi)}
\newcommand{\lz}{\left(}
\newcommand{\pz}{\right)}
\newcommand{\bfrac}[2]{\lz\frac{#1}{#2}\pz}
\newcommand{\odd}{\mathrm{\ primary}}
\newcommand{\even}{\text{ even}}
\newcommand{\res}{\mathrm{Res}}
\newcommand{\sumn}{\sumstar_{(c,1+i)=1}  w\left( \frac {N(c)}X \right)}
\newcommand{\lab}{\left|}
\newcommand{\rab}{\right|}
\newcommand{\Go}{\Gamma_{o}}
\newcommand{\Ge}{\Gamma_{e}}
\newcommand{\M}{\widehat}

\def\su#1{\sum_{\substack{#1}}}
\def\phis{\varphi^*}

\theoremstyle{plain}
\newtheorem{conj}{Conjecture}
\newtheorem{remark}[subsection]{Remark}

\newcommand{\pfrac}[2]{\left(\frac{#1}{#2}\right)}
\newcommand{\pmfrac}[2]{\left(\mfrac{#1}{#2}\right)}
\newcommand{\ptfrac}[2]{\left(\tfrac{#1}{#2}\right)}
\newcommand{\pMatrix}[4]{\left(\begin{matrix}#1 & #2 \\ #3 & #4\end{matrix}\right)}
\newcommand{\ppMatrix}[4]{\left(\!\pMatrix{#1}{#2}{#3}{#4}\!\right)}
\renewcommand{\pmatrix}[4]{\left(\begin{smallmatrix}#1 & #2 \\ #3 & #4\end{smallmatrix}\right)}
\def\en{{\mathbf{\,e}}_n}

\newcommand{\ppmod}[1]{\hspace{-0.15cm}\pmod{#1}}
\newcommand{\ccom}[1]{{\color{red}{Chantal: #1}} }
\newcommand{\acom}[1]{{\color{blue}{Alia: #1}} }
\newcommand{\alexcom}[1]{{\color{green}{Alex: #1}} }
\newcommand{\hcom}[1]{{\color{brown}{Hua: #1}} }

\makeatletter
\def\widebreve{\mathpalette\wide@breve}
\def\wide@breve#1#2{\sbox\z@{$#1#2$}%
     \mathop{\vbox{\m@th\ialign{##\crcr
\kern0.08em\brevefill#1{0.8\wd\z@}\crcr\noalign{\nointerlineskip}%
                    $\hss#1#2\hss$\crcr}}}\limits}
\def\brevefill#1#2{$\m@th\sbox\tw@{$#1($}%
  \hss\resizebox{#2}{\wd\tw@}{\rotatebox[origin=c]{90}{\upshape(}}\hss$}
\makeatletter

\title[Lower bounds for shifted moments of Dirichlet $L$-functions of fixed modulus]{Lower bounds for shifted moments of Dirichlet $L$-functions of fixed modulus}

\author[P. Gao]{Peng Gao}
\address{School of Mathematical Sciences, Beihang University, Beijing 100191, China}
\email{penggao@buaa.edu.cn}

\author[L. Zhao]{Liangyi Zhao}
\address{School of Mathematics and Statistics, University of New South Wales, Sydney NSW 2052, Australia}
\email{l.zhao@unsw.edu.au}

\begin{abstract}
We establish sharp lower bounds for shifted (with two shifts) moments of Dirichlet $L$-function of fixed modulus under the generalized Riemann hypothesis.
\end{abstract}

\maketitle



\section{Introduction}\label{sec 1}

Because of their significant implication in some very important analytic number theoretic problems, the moments of families of $L$-functions on the critical line have been a subject of sustained interest and intense investigation. In connections with random matrix theory, J. P. Keating and N. C. Snaith \cite{Keating-Snaith02} conjectured asymptotic formulas for moments of various families of $L$-functions. In \cite{CFKRS}, J. B. Conrey, D. W. Farmer, J. P. Keating, M. O. Rubinstein and N. C. Snaith made more precise predictions on the asymptotic behaviors of these moments with lower order terms.  A systematic method was developed by K. Soundararajan \cite{Sound2009} to obtain upper bounds for moments of $L$-functions under the generalized Riemann hypothesis (GRH). For instance, he considered moments of the Riemann zeta function $\zeta(s)$ on the critical line given by
\begin{align}
\label{zetamoment}
   I_k(T):= \int\limits^T_0|\zeta(\frac 12+it)|^{2k} \dif t,
\end{align}
    where $T$, $k>0$ are real numbers. The method of Soundararajan was further implemented by A. J. Harper \cite{Harper} to establish sharp upper bounds for moments conditionally. In particular, it applies to $I_k(T)$ in \eqref{zetamoment}. \newline

  In \cite{Radziwill&Sound}, a novel principle was introduced by M. Radziwi{\l \l} and K. Soundararajan to enable one to pursue upper bounds of all smaller moments with the knowledge of an upper bound for a particular moment. This approach was applied by W. Heap, M. Radziwi{\l\l} and K. Soundararajan \cite{HRS} to achieve sharp upper bounds for $I_k(T)$ unconditionally for all real $0 \leq k \leq 2$. Another application of this principle can be found in \cite{Curran}, in which M. J. Curran established under GRH upper bounds for the shifted moments
\begin{align}
\label{eqn:shiftedMomentszeta}
 \int\limits_T^{2T}  \prod_{k = 1}^m |\zeta(\tfrac{1}{2} + i (t + a_k))|^{2 b_k} \dif t,
\end{align}
    where $a_k, b_k$ are real numbers satisfying $|a_k| \leq T/2$ and $b_k\geq 0$. The result of Curran improves upon earlier ones given in \cites{Chandee11, NSW}. Also, the bounds obtained in \cites{Curran} for the expressions in \eqref{eqn:shiftedMomentszeta} are indeed sharp. In fact,
   using a dual lower bounds principle developed by W. Heap and K. Soundararajan \cite{H&Sound} as well as the work of Radziwi{\l \l} and  Soundararajan in \cite{Radziwill&Sound},  Curran \cite{Curran24} obtained matching lower bounds for the shifted moments of $\zeta(s)$. \newline

One significant aspect of the study of shifted moments was exhibited by M. Munsch \cite{Munsch17}, who obtained upper bounds for the shifted moments of the family of Dirichlet $L$-functions to a fixed modulus and applied the result to estimate moments of character sums. \newline

  Using an approach similar to that in \cite{Curran}, B. Szab\'o \cite{Szab} strengthened the result in \cite{Munsch17}, proving under GRH, for a large fixed modulus $q$, any  positive integer $k$, real tuples ${\bf a} =  (a_1, \ldots, a_k), {\bf t} =  (t_1, \ldots, t_k)$ with $a_j \geq 0$ and $|t_j| \leq q^A$ for a fixed positive real number $A$,
\begin{align}
\label{eqn:shiftedMoments}
M_{\bf{t}, \bf{a}}(q) := \sum_{\chi \in X_q^*} \big| L\big( \tfrac12+it_1,\chi \big) \big|^{2a_1} \cdots \big| L\big( \tfrac12+it_{k},\chi \big) \big|^{2a_{k}} \ll_{\bf{t}} &  \varphi(q)(\log q)^{a_1^2+\cdots +a_{k}^2} \prod_{1\leq j<l\leq k}  \big|\zeta(1+i(t_j-t_l)+\tfrac 1{\log q}) \big|^{2a_ja_l},
\end{align}
where $X_q^*$ denotes the set of primitive Dirichlet characters modulo $q$, $\varphi$ is the Euler totient function. We also write $\phis(q)$ for the cardinality of $X_q^*$. \newline

  Motivated by the above mentioned result of Curran \cite{Curran24}, the aim of this paper is to establish the corresponding lower bounds for $M_{\bf{t}, \bf{a}}(q)$. Our main result establishes the minorant for $k=2$.
\begin{theorem}
\label{thm:main}
With the notation as above and the truth of GRH, let $k=2$ and $|t_j| \leq q^A$ for $j = 1, \ldots , k$ for a fixed positive real number $A$. Assume also that  $\phis(q) \geq \eta \varphi(q)$ for a fixed constant $\eta>0$.  we have
\begin{align}
\begin{split}
\label{eqn:shiftedMomentslowerbounds}
M_{{\bf t}, \bf{a}}(q)  \gg_{\bf{t}} \varphi(q)(\log q)^{a_1^2+\cdots +a_{k}^2} \prod_{1\leq j<l\leq k}  \big|\zeta(1+i(t_j-t_l)+\tfrac 1{\log q}) \big|^{2a_ja_l}.
\end{split}
\end{align}
\end{theorem}

  We reserve the letter $p$ for a prime number throughout the paper. Note that (see \cite[(3.7)]{iwakow})  $\varphi^*(q)$ is a multiplicative function
  satisfying $\varphi^*(p)=p-2$ and $\varphi^*(p^m)=p^m(1-1/p)^2$ for integers $m \geq 2$. Note also that there are no primitive Dirichlet characters modulo $q$ when $q \equiv 2 \pmod 4$. Thus as a particular case of Theorem \ref{thm:main}, the lower bounds given in \eqref{eqn:shiftedMomentslowerbounds} are valid for primes $q \not\equiv 2 \pmod 4$. Also, the result in Theorem \ref{thm:main} shows that the bounds given in \eqref{eqn:shiftedMoments} and \eqref{eqn:shiftedMomentslowerbounds} are of the correct order of magnitude. \newline

   The proof of Theorem \ref{thm:main} is based on the method in \cite{Curran24} and is carried out for a general integer $k$ whenever possible. We shall indicate where the restriction $k=2$ is necessary.  We also point out here without further elaboration that the implicit constants in big-$O$ or Vinogradov asymptotic notation usually will depend on $\bf{a}$, although we shall often omit this in the subscripts.  Moreover, $q$ is implicitly assumed to be sufficiently large in terms of $\bf{a}$.

\section{Preliminary Results}

\subsection{Sums over primes}
\label{sec2.1}

  We note the following result concerning sums over primes.
\begin{lemma}
\label{RS} Let $x \geq 2$ and $\alpha \geq 0$. We have, for some constant $b_1$,
\begin{align}
\label{merten}
\sum_{p\le x} \frac{1}{p} =& \log \log x + b_1+ O\Big(\frac{1}{\log x}\Big), \\
\label{mertenpartialsummation}
\sum_{p\le x} \frac {\log p}{p} =& \log x + O(1),   \\
\label{mertenstype}
  \sum_{p\leq x} \frac{\cos(\alpha \log p) }{p}=& \log |\zeta(1+1/\log x+i\alpha)| +O(1).
\end{align}
\end{lemma}
\begin{proof}
  The expressions in \eqref{merten} and \eqref{mertenpartialsummation} can be found in parts (d) and (b) of \cite[Theorem 2.7]{MVa1}, respectively. The formula \eqref{mertenstype} is a special case of \cite[Lemma 3.2]{Kou}.
\end{proof}

\subsection{Upper Bounds for $\log |L(\tfrac12+it,\chi)|$}
\label{sec2.4}

We note the following variation of \cite[Proposition 2]{Szab} which bounds the sums of $\log |L(\tfrac12+it,\chi)|$ over various $t$ by Dirichlet polynomials.

\begin{lemma}
\label{lem:logZetaUpperBound}
Assume that GRH holds for $L(s,\chi)$, where $\chi$ is a primitive Dirichlet character modulo $q$. For a positive integer $k$, let $A,a_1,a_2,\ldots, a_{k}$ be positive constants and write $a =: a_1+\cdots+ a_{k}$.  Suppose also $t_1,\ldots, t_{k}$ are real numbers such that $|t_i|\leq q^A$. For any integer $n$, set
$$h(n) := \frac{1}{2}(a_1n^{-it_1}+\cdots +a_{k}n^{-it_{k}}).$$
Then, for large $q$ and $x \geq 2$,
\begin{align}
\label{logLboundgen}
\sum^k_{j=1}a_j\log |L(\tfrac12+it_j,\chi)|  \leq 2\cdot\Re \sum_{p\leq x} \frac{h(p)\chi(p)}{p^{1/2+1/\log x}}\frac{\log (x/p)}{\log x}+\Re\sum_{p\leq x^{1/2}} \frac{h(p^2)\chi(p^2)}{p}+a(A+1)\frac{\log q}{\log x}+O(1) .
\end{align}  

   Moreover, for any non-quadratic character $\chi$, we have 
\begin{align}
\label{logLboundnonquad}
\sum^k_{j=1}a_j\log |L(\tfrac12+it_j,\chi)|  \leq 2\cdot\Re \sum_{p\leq x} \frac{h(p)\chi(p)}{p^{1/2+1/\log x}}\frac{\log (x/p)}{\log x}+\Re\sum_{p\leq \min (\log q, x^{1/2})} \frac{h(p^2)\chi(p^2)}{p}+a(A+1)\frac{\log q}{\log x}+O(1) .
\end{align}
\end{lemma}
\begin{proof}
  We note that the estimation given in \eqref{logLboundgen} follows from \cite[Proposition 2]{Szab}. We may now assume that $\log q \leq x^{1/2}$ and we note from \cite[Lemma 2.5]{G&Zhao24-12} that for any primitive Dirichlet character $\chi$ modulo $q$ and any $t_0 \in \mr$, we have under GRH, 
\begin{align}
\label{PIT}
 \sum_{p \leq x } \chi(p) p^{-it_0}\log p  = &O(\sqrt{x} \left(\log 2q(x+|t_0|))^2 \right), \quad \chi \neq \chi_0, 
\end{align}
 where $\chi_0$ denotes the principal character modulo $q$. \newline
 
   We apply \eqref{PIT} and partial summation to see that
\begin{align}
\label{sumplarge}
 \sum_{(\log q)^6 < p\leq x^{1/2}} \frac{h(p^2)\chi(p^2)}{p}   = O(1).
\end{align}
 We then apply \eqref{merten} to see that
\begin{align}
\label{sumpsmall}
  \sum_{\substack{\log q < p  \leq (\log q)^6}}  \frac{h(p^2)\chi(p^2)}{p} \ll \sum_{\substack{\log q  \leq q \leq (\log q)^6}}
  \frac{1}{p}  =  O(1).
\end{align}

Applying \eqref{sumplarge}--\eqref{sumpsmall} in \eqref{logLboundgen} now leads to the estimation given in \eqref{logLboundnonquad}. This completes the proof of the lemma.
\end{proof}

\subsection{Approximate functional equation}

   For any Dirichlet character $\chi$ modulo $q$, the associated Gauss sum  $\tau(\chi)$ is defined to be
\begin{align*}
  \tau(\chi)=\sum^q_{n=1}\chi(n)e \Big( \frac {n}{q} \Big).
\end{align*}

  If $\chi$ is also primitive, then $L(s, \chi)$ satisfies the functional equation (see \cite[\S 9]{Da}) given by
\begin{align}
\label{fneqnquad}
  \Lambda(s, \chi) := \Big( \frac {q} {\pi} \Big)^{(s+\mathfrak{a})/2}\Gamma \Big( \frac {s+\mathfrak{a}}{2} \Big)L(s, \chi)=
  \frac {i^{\mathfrak{a}}q^{1/2}}{\tau(\chi)}\Lambda(1-s,  \overline{\chi}),
\quad \mbox{where} \quad
  \mathfrak{a}=
\begin{cases}
  0, \quad \mbox{if} \;\; \chi(-1)=1, \\
  1,  \quad \mbox{if} \;\; \chi(-1)=1.
\end{cases}
\end{align}

Now, \eqref{fneqnquad} allows us to deduce the following approximate functional equation for $\prod_{m\leq k} L(\tfrac{1}{2} + i t_m, \chi)$.
\begin{lemma}
\label{lem:AFEpure}
 Let $\chi$ be a primitive Dirichlet character modulo $q$, $k$ a positive integer and ${\bf t}=(t_1,\ldots, t_{k})$ a $k$-tuple of real numbers.  Set
\begin{equation} \label{Wdef}
W_{\mathfrak{a}, \pm {\bf t}}(x) = \frac{1}{2\pi i} \int\limits_{(3)} \frac{e^{s^2}}{s} \left(x\pi^{k/2} \right)^{-s}\prod_{1 \leq m\leq k} \Gamma \Big( \frac {1/2\pm it_m+s+\mathfrak{a}}{2}\Big) \Gamma \Big( \frac {1/2+it_m+\mathfrak{a}}{2}\Big)^{-1}  \dif s,
\end{equation}
and $\tau_{{\bf t}}(n) = \sum_{n_1\cdots n_k = n} n_1^{-it_1} \cdots n_k^{-it_m}$, where $n$, $n_i$, $1\leq i \leq k$ are positive integers.
For any real $X>0$, we have
\begin{align}
\label{Approxfcneqn}
\prod_{1 \leq m\leq k} L(\tfrac{1}{2} + i t_m, \chi) = \sum_n \frac{\tau_{{\bf t}}(n)\chi(n)}{n^{1/2}} W_{\mathfrak{a}, {\bf t}} \Big( \frac {nX}{q^{k/2}} \Big) + \big (\frac {i^{\mathfrak{a}}q^{1/2}}{\tau(\chi)}\big )^k\Big( \frac {q} {\pi} \Big)^{-i\sum_{m \leq k}t_m}\sum_n \frac{\tau_{-{\bf t}}(n)\overline\chi(n)}{n^{1/2}} W_{\mathfrak{a}, -{\bf t}} \Big( \frac {n}{Xq^{k/2}} \Big).
\end{align}
\end{lemma}
\begin{proof}
We start with the integral
\begin{align}
\label{Idef}
I(X) = \frac{1}{2\pi i} \int\limits_{(3)} X^{-s}\frac{e^{s^2}}{s}  \prod_{1 \leq m\leq k} \Lambda(\tfrac{1}{2} + it_m + s, \chi) \dif s.
\end{align}

  We shift the contour of integration in \eqref{Idef} to $\Re s = -3$, encountering only a simple pole at $s = 0$, with residue $\prod_{m\leq k} \Lambda(\tfrac{1}{2} + i t_m, \chi)$.  It follows that
\[
I(X) = \prod_{m\leq k} \Lambda(\tfrac{1}{2} + i t_m, \chi) + \frac{1}{2\pi i} \int\limits_{(-3)}X^{-s}\frac{e^{s^2}}{s} \prod_{m\leq k} \Lambda(\tfrac{1}{2} + it_m + s, \chi) \dif s.
\]

Applying \eqref{fneqnquad} yields
\[
I(X) = \prod_{m\leq k} \Lambda(\tfrac{1}{2} + i t_m, \chi) + \Big( \frac {i^{\mathfrak{a}}q^{1/2}}{\tau(\chi)} \Big)^k\frac{1}{2\pi i} \int\limits_{(-3)}X^{-s}\frac{e^{s^2}}{s} \prod_{m\leq k} \Lambda(\tfrac{1}{2} - it_m - s, \overline \chi) \dif s.
\]

A change of variable $s \rightarrow -s$ in the last integral above leads to
\[
I(X) = \prod_{m\leq k} \Lambda(\tfrac{1}{2} + i t_m, \chi) - \Big (\frac {i^{\mathfrak{a}}q^{1/2}}{\tau(\chi)}\Big )^k\frac{1}{2\pi i} \int\limits_{(3)}X^{s}\frac{e^{s^2}}{s} \prod_{m\leq k} \Lambda(\tfrac{1}{2} - it_m + s,  \overline\chi) \dif s.
\]

Hence
\begin{align}
\label{lambdaexp}
\begin{split}
\prod_{m\leq k} \Lambda(\tfrac{1}{2} + i t_m, \chi) = \frac{1}{2\pi i}\int\limits_{(3)}X^{-s}\frac{e^{s^2}}{s} \prod_{m\leq k} \Lambda(\tfrac{1}{2} + it_m + s, \chi) \dif s+\big (\frac {i^{\mathfrak{a}}q^{1/2}}{\tau(\chi)}\big )^k\frac{1}{2\pi i}    \int\limits_{(3)} X^s\frac{e^{s^2}}{s}  \prod_{m\leq k} \Lambda(\tfrac{1}{2} - it_m + s, \overline \chi) \dif s.
\end{split}
\end{align}

 We now expand the products in the integrands of \eqref{lambdaexp} into Dirichlet series and integrate term by term, getting
\begin{align*}
\begin{split}
\frac{1}{2\pi i} & \int\limits_{(3)}X^{-s}\frac{e^{s^2}}{s} \prod_{m\leq k} \Lambda(\tfrac{1}{2} + it_m + s, \chi) \dif s \\
&=  \sum_n \frac{\tau_{{\bf t}}(n)\chi(n)}{n^{1/2}}  \frac{1}{2\pi i} \int\limits_{(3)} \frac{e^{s^2}}{s} \left(\frac{q^{k/2}}{nX\pi^{k/2}}\right)^s\prod_{m\leq k}\Big( \frac {q} {\pi} \Big)^{(1/2+it_m+\mathfrak{a})/2}\Gamma \Big( \frac {1/2+it_m+s+\mathfrak{a}}{2} \Big) \dif s, \quad \mbox{and} \\
 \frac{1}{2\pi i} & \int\limits_{(3)}X^{s}\frac{e^{s^2}}{s} \prod_{m\leq k} \Lambda(\tfrac{1}{2} - it_m + s, \overline  \chi) \dif s \\
& = \sum_n \frac{\tau_{-{\bf t}}(n)\overline \chi(n)}{n^{1/2}}  \frac{1}{2\pi i} \int\limits_{(3)} \frac{e^{s^2}}{s} \left(\frac{Xq^{k/2}}{n\pi^{k/2}}\right)^s\prod_{m\leq k}\Big( \frac {q} {\pi} \Big)^{(1/2-it_m+\mathfrak{a})/2}\Gamma \Big( \frac {1/2-it_m+s+\mathfrak{a}}{2} \Big) \dif s.
\end{split}
\end{align*}

Inserting the above expressions into \eqref{lambdaexp} and then divide on both sides of the resulting formulaby  $$\prod_{m\leq k}\Big( \frac {q} {\pi} \Big)^{(1/2+it_m+\mathfrak{a})/2}\Gamma \Big( \frac {1/2+it_m+\mathfrak{a}}{2} \Big)$$
gives the desired expression in \eqref{Approxfcneqn}, completing the proof of the lemma.
\end{proof}

    Note that shifting the contour in \eqref {Wdef} to $\Re s = -1/4$ or to $\Re s =c$ for any $c>0$ renders that for any real $x>0$, 
\begin{align}
\label{Westimation}
 W_{\pm 1, {\bf t}}(x)=\begin{cases}
 C(\pm 1, {\bf t})+O(x^{1/4}), \\
 O_c(x^{-c}),
\end{cases} 
\end{align}  
  where $C(\pm 1, {\bf t})$ are constants depending on $\pm 1$ and ${\bf t}$. 

\subsection{Estimations on Dirichlet polynomials}

 Let $k$ be a positive integer. For a $k$-tuple of positive constants ${\bf a}=(a_1,a_2,\ldots, a_{k})$, write
\begin{equation}
\label{aastdef}
 a_\ast := \sum_{j\leq k} \max(1, a_j).
\end{equation}
For any primitive Dirichlet character $\chi$ modulo $q$, we define a sequence of parameters $P_j = q^{c_j}$ by
\[ c_0 = 0 \quad \text{and} \quad c_j = \frac{e^{j}}{(\log\log q)^2}, \quad j>0. \]

Let $R$ be the largest integer such that $P_R \leq q^{\delta}$ for some constant $0 < \delta < e^{-1000a_\ast}$ depending on $\bf{a}$ to be chosen later.
We write for any $x>0$,
\[ \mathcal{P}_{1,x}(s,\chi) = \sum_{p \leq P_1} \frac{\chi(p)}{p^{s+1/\log x}} \frac{\log (x/p)}{\log x} + \sum_{p\leq \log q} \frac{\chi(p^2)}{2p^{2s}} \quad \mbox{and} \quad \mathcal{P}_{j,x}(s, \chi) = \sum_{p\in (P_{j-1},P_j]} \frac{\chi(p)}{p^{s+1/\log x}} \frac{\log (x/p)}{\log x}, \quad 2\leq j \leq R. \]

   Let $K_j = c_j^{-3/4}$ for $1 \leq j \leq R$.  Define for any real number $\beta$, a Dirichlet polynomial $\mathcal{N}_{j,x}(s, \chi;\beta)$ by
\[\label{eq:TaylorExp}
\mathcal{N}_{j,x}(s, \chi;\beta) := \sum_{1 \leq n \leq 100a_\ast^2 K_j} \frac{\beta^n \mathcal{P}_{j,x}(s, \chi)^n}{n!}.
\]

  By a straightforward modification of the proof of \cite[Lemma 2.2]{Curran24},  we arrive at the following result which may be summarized as follows.  For $\mathcal{P}_{j,x}(s, \chi)$ not too large, one can approximate $\exp(\beta \mathcal{P}_{j,x} (s, \chi))$ by $\mathcal{N}_{j,x}(s, \chi;\beta)$.
\begin{lemma}
\label{lem:expTaylorSeries}
If $\beta \leq a_\ast$ and $|\mathcal{P}_{j,x}(s, \chi)| \leq K_j$ for some $1\leq j \leq R$, then
\[
\exp(\beta \mathcal{P}_{j,x}(s, \chi)) = (1+ O(e^{-50a_\ast^2 K_j}))^{-1} \mathcal{N}_{j,x}(s,\chi;\beta).
\]
\end{lemma}

Let $a_x(p)$ denote $\log(x/p)  p^{-1/\log x}/\log x$ and define multiplicative functions $g_x$ and $h_x$ by
\[
g_x(p^r; a) :=  \frac{a^r a_x(p)^{r}}{r!} \quad \mbox{and} \quad
h_x(p^r;a) :=  g_x(p^r; a) + {\bf 1}_{p\leq \log q} \sum_{u = 1}^{r/2} \frac{a^{r-2u} a_x(p)^{r-2u}}{2^u u! (r-2u)!} .
\]
We further define $c_j(n) : \natn \to \{ 0, 1 \}$ for $1 \leq j \leq R$.  Set $c_{1}(n)=1$ if and only if $n$ can be written as $n = n_1 \cdots n_r$ where $r \leq 100a_\ast^2 K_1$ and each $n_i$ is either a prime $\leq P_1$ or a prime square $\leq \log q$.  For other $c_j(n)$ with $j \geq 2$, $c_j(n)=1$ if and only if $n$ is the product of at most $100a_\ast^2 K_j$ (not necessarily distinct) primes in $(P_{j-1},P_j]$. \newline

Using the above notations and \cite[proposition 3.1]{Curran}, we deduce the next lemma concerning the coefficients $\mathcal{N}_{j,x}(s, \chi;\beta)$.
\begin{prop}\label{prop:Njcoeffs}
 With the notation as above, we have for $2\leq j \leq R$,
\[
\mathcal{N}_{j,x}(s, \chi;\beta) = \sum_{p\mid n \Rightarrow p \in (P_{j-1},P_j] } \frac{g_x(n;\beta)\chi(n)c_{j}(n)}{n^s}.
\]
Furthermore, if
\[
\mathcal{N}_{1,x}(s, \chi;\beta) = \sum_{p\mid n \Rightarrow p \in (P_{j-1},P_j]} \frac{f_x(n;\beta)\chi(n)}{n^s},
\]
then $f_x(n;\beta) \leq h_x(n;\beta) c_{1}(n)$ and $f_x(p;\beta) = g_x(p;\beta)$.
\end{prop}
\noindent

  We now consider the coefficients of the Dirichlet polynomials given by
\[
\prod_{1 \leq m\leq k}  \mathcal{N}_{j,x}(s+ i t_m, \chi; a_m) := \sum_{n} \frac{b_{j,x,\bf{t},\bf{a}}(n)\chi(n)}{n^{s}}.
\]

  It follows from our discussions above that $b_{1,x,\bf{t},\bf{a}}(n)$ is the $k$-fold Dirichlet convolution of $f_x(n;a_m) n^{-it_m}$ and
  $b_{j,x,\bf{t},\bf{a}}(n)$ is the $k$-fold convolution of $g_x(n;a_m) c_{j}(n) n^{-it_m}$ for $2 \leq j \leq R$.
We further define $b'_{j,x,\bf{t},\bf{a}}(n)$ to be the $k$-fold convolution of $h_x(n;a_m) n^{-it_m} {\bf 1}_{p|n \Rightarrow p\in (P_0,P_1]}$ for $j=1$ and
the $k$-fold  convolution  of $g_x(n;a_m) n^{-it_m} {\bf 1}_{p|n \Rightarrow p\in (P_{j-1},P_j]}$ for $2\leq j \leq R$.
Also, let $b''_{1,x,\bf{t},\bf{a}}(n)$ be the $k$-fold  convolution of $h_x(n;a_m) {\bf 1}_{p|n \Rightarrow p\in (P_0,P_1]}$ when $j= 1$ and the $k$-fold
convolution of $g_x(n;a_m) {\bf 1}_{p|n \Rightarrow p\in (P_{j-1},P_j]}$ when  $2\leq j \leq R$.
It follows that both $b'_{j,x,\bf{t},\bf{a}}$ and $b''_{j,x,\bf{t},\bf{a}}$ are multiplicative and we have $|b_{j,x,\bf{t},\bf{a}}(n)|$, $|b'_{j,x,\bf{t},\bf{a}}(n)| \leq b''_{j,x,\bf{t},\bf{a}}(n)$.  From \cite[Lemma 3.2]{Curran}, we gather some information on these coefficients.
\begin{lemma}\label{lem:nCoeffPatrol}
For $1\leq j \leq R$ and $p \in (P_{j-1},P_j]$, we have
\[
b_{j, x, \bf{t}, \bf{a}}(p) = a_x(p) \sum_{m = 1}^k a_m p^{-it_m},
\]
and $b''_{j, x, \bf{t}, \bf{a}}(p) \leq a_\ast$. For $r \geq 2$, the bound
\[
b''_{j, x, \bf{t}, \bf{a}}(p^r) \leq \frac{a_\ast^r k^r}{r!}
\]
holds whenever $2\leq j \leq R$ or $p > \log q$, and otherwise
\[
b''_{1, x, \bf{t}, \bf{a}}(p^r) \leq  k a_\ast^r r^{2k} e^{-r \log (r/k)/2k + 2r}.
\]
\end{lemma}

\section{Proof of Theorem \ref{thm:main}}

   We first note that upon setting $k=a_1=1, x=\log q$ in \eqref{logLboundgen} and estimating the summations there trivially implies that for some constant $C>0$, 
\begin{align}
\label{logLupperboundgen}
& |L(\frac12+it, \chi)|  \ll \exp(\frac {C\log q}{\log \log q}). 
\end{align}   
  On the other hand, we note that (see \cite[Theorem 2.9]{MVa1}) for $q \geq 3$, 
\begin{align*}
& \varphi(q) \gg \frac {q}{\log \log q}. 
\end{align*}  

   It follows from the above that in order to establish Theorem \ref{thm:main}, we may ignore the presence of the finitely many quadrati Dirichlet character. In particular, we may assume that \eqref{logLboundnonquad} holds for all Dirichlet characters throughout the proof. 
   
  As mentioned in the Section \ref{sec 1}, $k$ is an arbitrary (but fixed) positive integer throughout the proof unless otherwise specified. We define
\begin{align*}
\mathcal{S}_0 :=& \sum_{\chi \in \mathcal{T}_k} \prod_{1 \leq m\leq k} L(\tfrac{1}{2} + it_m, \chi) \prod_{1 \leq j\leq R} \exp\left((a_m-1)
\mathcal{P}_{j,P_R}(\tfrac{1}{2}+it_m, \chi) + a_m  \mathcal{P}_{j,P_R}(\tfrac{1}{2}-it_m, \overline \chi) \right) ,
\end{align*}
  where
\begin{equation*}
\mathcal{T}_k := \left\{\chi\in X_q^*: |P_{j,P_R} (\tfrac{1}{2} + i t_m, \chi)| \leq K_j \text{ for all } 1\leq j\leq R, 1 \leq m \leq k \right\}.
\end{equation*}
 Note, mindful of Lemma~\ref{lem:expTaylorSeries}, that $\mathcal{T}_k$ is the set of characters $\chi\in X_q^*$ such that $\mathcal{N}_{j,P_R}(\tfrac{1}{2} + i t_m, \chi;\beta)$  well approximates $\exp(\beta \mathcal{P}_{j,P_R} (\tfrac{1}{2} + i t_m, \chi))$ for all $1 \leq m \leq k$.  We also note that $\mathcal{S}_0$ is regarded as an approximation to $M_{\bf{t}, \bf{a}}(q)$. \newline

  We define real numbers $u, v, r_m, 1\leq m \leq k$ such that
\[
\frac{1}{u} = \frac{1}{4a_\ast}, \quad \frac{1}{r_m} = \frac{1}{2k} - \frac{a_m}{4ka_\ast}, \quad \frac{1}{v} = 1 - \frac{1}{u} - \sum_{m\leq k} \frac{1}{r_m}.
\]
We have $u$, $v$, $r_m > 1$ for all $1\leq m \leq k$ based on the definition of $a_\ast$ in \eqref{aastdef}. Moreover,
\[
\frac{2a_m}{u}+\frac {2k}{r_m} = \frac{1}{u}+\frac{1}{v} + \sum_{m\leq k} \frac{1}{r_m}=1.
\]

 We now recast $\mathcal{S}_0$ as
\begin{align*}
\begin{split}
\mathcal{S}_0 =& \sum_{\chi \in \mathcal{T}_k} \Big (\prod_{1 \leq m\leq k} L(\tfrac{1}{2} + it_m, \chi)^{2a_m/u} \Big ) \Big (\prod_{1 \leq m\leq k}\prod_{1 \leq j\leq R} \exp\Big(\frac {2a_m}v \Re \mathcal{P}_{j,P_R}(\tfrac{1}{2}+it_m,
\chi) \Big)\Big ) \\
&  \hspace*{1cm} \times \prod_{1 \leq m\leq k} \left( L(\tfrac{1}{2} + it_m, \chi)^{2k/r_m}\prod_{1 \leq j\leq R} \exp\left(2a_m\Big( 1-\frac 1v \Big) \Re \mathcal{P}_{j,P_R}(\tfrac{1}{2}+it_m, \chi)-\mathcal{P}_{j,P_R}(\tfrac{1}{2}+it_m, \chi) \right) \right) \\
=& \sum_{\chi \in \mathcal{T}_k} \Big (\prod_{1 \leq m\leq k} L(\tfrac{1}{2} + it_m, \chi)^{2a_m/u} \Big ) \left (\prod_{1 \leq m\leq k}\prod_{1 \leq j\leq R} \exp\left(\frac {2a_m}v \Re  \mathcal{P}_{j,P_R}(\tfrac{1}{2}+it_m,
\chi) \right)\right ) \\
& \times  \prod_{1 \leq m\leq k} L(\tfrac{1}{2} + it_m, \chi)^{2k/r_m} \\
& \times \prod_{1 \leq m\leq k} \prod_{1 \leq j\leq R} \exp\left(\Big(\sum_{1 \leq s\leq k} \frac{1}{r_s}\Big)2a_m \Re \mathcal{P}_{j,P_R}(\tfrac{1}{2}+it_m, \chi)- \mathcal{P}_{j,P_R}(\tfrac{1}{2}+it_m, \chi)+ \frac{2a_m}{u} \Re \mathcal{P}_{j,P_R}(\tfrac{1}{2}+it_m, \chi)\Big)\right) \\
=& \sum_{\chi \in \mathcal{T}_k} \Big (\prod_{1 \leq m\leq k} L(\tfrac{1}{2} + it_m, \chi)^{2a_m/u} \Big ) \Big (\prod_{1 \leq m\leq k}\prod_{1 \leq j\leq R} \exp\Big(\frac {2a_m}v \Re \mathcal{P}_{j,P_R}(\tfrac{1}{2}+it_m,
\chi) \Big)\Big ) \\
& \hspace*{1cm} \times \prod_{1 \leq m\leq k} \Big ( L(\tfrac{1}{2} + it_m, \chi)^{2k/r_m} \prod_{1 \leq \ell \leq k} \prod_{1 \leq j\leq R} \exp\Big(\frac{2a_{\ell}}{r_m} \Re \mathcal{P}_{j,P_R}(\tfrac{1}{2}+it_{\ell}, \chi) \Big)  \\
& \hspace*{1cm} \times \prod_{1 \leq j\leq R}\exp\Big( - \mathcal{P}_{j,P_R}(\tfrac{1}{2}+it_m, \chi)+ \frac{2a_m}{u} \Re \mathcal{P}_{j,P_R}(\tfrac{1}{2}+it_m, \chi)\Big) \Big).
\end{split}
\end{align*}

Now Hölder's inequality applied the last expression above leads to
\begin{equation} \label{S0bound}
|\mathcal{S}_0| \leq M_{\bf{t}, \bf{a}}(q)^{1/u} \times |\mathcal{J}|^{1/v} \times \prod_{1 \leq m \leq k} |\mathcal{S}_m|^{1/r_m},
\end{equation}
where
\begin{align*}
\begin{split}
\mathcal{S}_m =& \sum_{\chi \in \mathcal{T}_k} |L(\tfrac{1}{2} + it_m, \chi)|^{2k} \prod_{1 \leq j\leq R} \exp\left(2(a_m-k)\Re \mathcal{P}_{j,P_R}(\tfrac{1}{2}+it_m, \chi) \right)
\prod_{\substack{1 \leq \ell \leq k \\ \ell \neq m}} \prod_{1 \leq j\leq R} \exp\left(2a_\ell\Re \mathcal{P}_{j,P_R}(\tfrac{1}{2}+it_\ell, \chi) \right), \quad \mbox{and} \\
\mathcal{J} =& \sum_{\chi \in \mathcal{T}_k} \prod_{1 \leq m\leq k}\prod_{1 \leq j\leq R} \exp\left(2a_m\Re \mathcal{P}_{j,P_R}(\tfrac{1}{2}+it_m,
\chi) \right).
\end{split}
\end{align*}

  We see from \eqref{S0bound} that in order to prove Theorem \ref{thm:main}, it suffices to establish the following three propositions.
\begin{prop}\label{prop:JBound}
 We have for large $q$'s,
\[
\mathcal{J} \ll_{\bf{a}}  \varphi(q) (\log q)^{a_1^2 + \cdots + a_k^2} \prod_{1\leq j < l \leq k} |\zeta(1 + i(t_j - t_l) + 1/ \log q )|^{2a_ja_l}.
\]
\end{prop}

\begin{prop}\label{prop:I0Bound}
Assuming the truth of GRH,  we have for large $q$'s and $k=2$, 
\[
|\mathcal{S}_0| \gg_{\bf{a}}  \varphi(q) (\log q)^{a_1^2 + \cdots + a_k^2} \prod_{1\leq j < l \leq k} |\zeta(1 + i(t_j - t_l) + 1/ \log q )|^{2a_ja_l}.
\]
\end{prop}

\begin{prop}\label{prop:IkBound}
Assuming the truth of GRH,  we have for large $q$'s and $1\leq m \leq k$,
\[
\mathcal{S}_m \ll_{\bf{a}}  \varphi(q) (\log q)^{a_1^2 + \cdots + a_k^2} \prod_{1\leq j < l \leq k} |\zeta(1 + i(t_j - t_l) + 1/ \log q )|^{2a_ja_l}.
\]
\end{prop}

The remainder of the paper is devoted to the proofs of these propositions. 
  
\section{Proof of Proposition \ref{prop:JBound}}
\label{sect:J}
  
   Recall that
\begin{align*}
\mathcal{J} = \sum_{\chi \in \mathcal{T}_k} \prod_{1 \leq m\leq k}\prod_{1 \leq j\leq R} \exp\left(2a_m\Re \mathcal{P}_{j,P_R}(\tfrac{1}{2}+it_m, \chi) \right).
\end{align*}
 We extend the $\chi$-sum above to $\chi \in X_q$, the set of all Dirichlet characters modulo $q$.  Applying Lemma \ref{lem:expTaylorSeries} while noting that $\prod_{1 \leq j\leq R} (1 +O(e^{-50a_\ast^2 K_j}))^2 = O(1)$ gives us
\begin{align*}
\mathcal{J} \ll \sum_{\chi \in X_q} \prod_{1 \leq m\leq k} |\mathcal{N}_{P_R}(\tfrac{1}{2}+it_m, \chi ; a_m)|^2 ,
\end{align*}
  where for any real number $\beta$, we define
\[
\mathcal{N}_x(s, \chi;\beta) := \prod_{1 \leq j\leq R} \mathcal{N}_{j,x}(s, \chi;\beta).
\]

Observe that $\mathcal{N}_{j,x}(s, \chi;a)$ is a Dirichlet polynomial of length at most $P_j^{100a_\ast^2 K_j}$ when $j > 1$ and at most $\leq P_1^{200a_\ast^2 K_1}$ for $j=1$. It follows from the choice of  $P_j$, $K_j$ and $R$ that, upon choosing $\delta$ small enough, 
\[
\prod_{1 \leq m\leq k} \mathcal{N}_{x}(s, \chi; a_m)
\]
has length at most $P_1^{200ka_\ast^2 K_1} P_2^{100ka_\ast^2 K_2} \cdots P_R^{100ka_\ast^2 K_R} \leq q^{2/10^M}$, where $M$ is a large natural number.
Observe also the orthogonality relation for Dirichlet characters (see \cite[Corollary 4.5]{MVa1}) asserts that for $(n,q)=1$,
\begin{align}
\label{orthrel}
\sum_{\chi \in X_q} \chi(n)=\begin{cases}
 \varphi(q), & \text{if} \ n \equiv 1 \pmod q, \\
 0, & \text{otherwise}.
\end{cases}
\end{align}

 We then proceed as in the proof of \cite[proposition 3.3]{Curran} while applying the well-known bound
\begin{align}
\begin{split}
\label{1xbound}
 1+x \leq e^x, \quad \text{for all} \; x \in \rear, 
\end{split}
\end{align} 
  to see that
\begin{align}
\begin{split}
\label{Tupperbound}
\mathcal{J} \ll & \varphi(q) \prod_{1 \leq j \leq R} \left(\prod_{p\in (P_{j-1},P_j]} \left(1 + \frac{|b_{j, P_R, \bf{t}, \bf{a}}(p)|^2}{p} + O\left(\frac{1}{p^2}\right)\right) + O(e^{-50a_\ast^2 K_j}) \right)\\
\ll & \varphi(q) \prod_{1 \leq j \leq R} \left(\prod_{p\in (P_{j-1},P_j]} \left(1 + \frac{|\sum_{m = 1}^k a_m p^{-it_m}|^2}{p} \right) \right) 
\ll \varphi(q) \prod_{p \leq q} \left(1 + \frac{|\sum_{m = 1}^k a_m p^{-it_m}|^2}{p} \right)  \\
\ll & \varphi(q)  \exp \left (\sum_{p \leq q} \frac{|\sum_{m = 1}^k a_m p^{-it_m}|^2}{p}  \right ) \ll \varphi(q) (\log q)^{a_1^2 + \cdots + a_k^2} \prod_{1\leq j<l \leq k} |\zeta(1 + i(t_j - t_l) + 1/ \log q )|^{2a_j a_l},
\end{split}
\end{align}
where the last estimation follows from  \eqref{merten} and \eqref{mertenstype}.
 This completes the proof of the proposition.

\section{Proof of Proposition \ref{prop:I0Bound}}
\subsection{Initial Treatments}

Recall that
\begin{align*}
\mathcal{S}_0 &= \sum_{\chi \in \mathcal{T}_k} \prod_{1 \leq m\leq k} L(\tfrac{1}{2} + it_m, \chi)\prod_{1 \leq j\leq R} \exp\left((a_m-1) \mathcal{P}_{j,P_R}(\tfrac{1}{2}+it_m, \chi) + a_m  \mathcal{P}_{j,P_R}(\tfrac{1}{2}-it_m, \overline \chi) \right).
\end{align*}

  We now recast $\mathcal{S}_0$ as a telescoping sum
\begin{align}
\begin{split}
\label{S0decomp}
\mathcal{S}_0 = \mathcal{S}_0^{(R)} - \sum_{1 \leq J \leq R} \left(\mathcal{S}_0^{(J)} - \mathcal{S}_0^{(J-1)}\right),
\end{split}
\end{align}
  where we define for $0 \leq J \leq R$,
\begin{align*}
  \mathcal{S}_0^{(J)} = \sum_{\chi \in \mathcal{T}_k} \prod_{1 \leq m\leq k} & L(\tfrac{1}{2} + it_m, \chi)\prod_{1 \leq j\leq J} \mathcal{N}_{j,P_R}(\tfrac{1}{2} + it_m, \chi; \tfrac{1}{2}(a_m - 1))^2 \overline{\mathcal{N}_{j,P_R}(\tfrac{1}{2} + it_m, \chi; \tfrac{1}{2}a_m)}^2 \\
\times & \prod_{J < j\leq R} \exp\left((a_m-1) \mathcal{P}_{j,P_R}(\tfrac{1}{2}+it_m, \chi) + a_m  \mathcal{P}_{j,P_R}(\tfrac{1}{2}-it_m, \overline \chi) \right).
\end{align*}
As usual, the empty product equals $1$. \newline
  
  We first bound $\mathcal{S}_0^{(R)}$ by  decomposing it as $\mathcal{S}_0^{(R)} = \mathcal{J}_1 - \mathcal{J}_2$, where
\begin{align*}
\begin{split}
\mathcal{J}_1 =& \sum_{\chi \in X^*_q} \prod_{1 \leq m\leq k} L(\tfrac{1}{2} + it_k, \chi) \mathcal{N}_{P_R}(\tfrac{1}{2} + it_m, \chi; \tfrac{1}{2}(a_m - 1))^2 \overline{\mathcal{N}_{P_R}(\tfrac{1}{2} + it_m, \chi; \tfrac{1}{2}a_m)}^2, \quad \mbox{and} \\
\mathcal{J}_2 =& \sum_{\chi \in X^*_q \setminus \mathcal{T}_k} \prod_{1 \leq m\leq k} L(\tfrac{1}{2} + it_m, \chi) \mathcal{N}_{P_R}(\tfrac{1}{2} + it_m, \chi;
 \tfrac{1}{2}(a_m - 1))^2 \overline{\mathcal{N}_{P_R}(\tfrac{1}{2} + it_m, \chi; \tfrac{1}{2}a_m)}^2.
\end{split}
\end{align*}

The following propositions estimate the sizes of $\mathcal{J}_1$ and $\mathcal{J}_2$. 
\begin{prop}
\label{prop:J1Lower}
We have
\[
\mathcal{J}_1 \gg  \varphi(q)\prod_{p \leq P_R} \left(1 + \sum_{1\leq j , m \leq k} \frac{a_j a_m}{p^{1+ i(t_j - t_m)}}\right).
\]
\end{prop}

\begin{prop}
\label{prop:J2J1comp}
Assuming the truth of GRH, we have, for $\delta$ being sufficiently small in terms of $\bf{a}$, 
\[
|\mathcal{J}_2| \leq \frac {|\mathcal{J}_1|}{2}.
\]
\end{prop}

 The above two propositions readily imply the following estimation on $\mathcal{S}_0^{(R)}$ .

\begin{prop}\label{prop:I0Llower}We have
\[ 
|\mathcal{S}_{0}^{(R)}| \gg  \varphi(q) \prod_{p\leq P_R} \left(1 + \sum_{1\leq j,m \leq k} \frac{a_j a_m}{p^{1 + i(t_j-t_m)}} \right).
\]
\end{prop}

  We now turn our attention to $\big|\mathcal{S}_0^{(J)} - \mathcal{S}_0^{(J-1)} \big|$. As $|\mathcal{P}_{j,P_R}(\tfrac{1}{2}+it_m,  \chi)| \leq K_j$ for all $1 \leq j \leq R, 1 \leq m \leq k$ by the definition of $\mathcal{T}_k$,  we deduce readily from Lemma \ref{lem:expTaylorSeries} the following bounds for $\big|\mathcal{S}_0^{(J)} - \mathcal{S}_0^{(J-1)} \big|$.
\begin{lemma}\label{lem:incrementBound}
We have for $1 \leq J \leq R$,
\begin{align}
\label{Sdiffest}
&\big|\mathcal{S}_0^{(J)} - \mathcal{S}_0^{(J-1)} \big|
\ll e^{-50a_\ast^2 K_J} \sum_{\chi \in \mathcal{T}_k} \prod_{1 \leq m\leq k} |L(\tfrac{1}{2} + it_m, \chi)| \cdot  |\mathcal{N}_{P_R}(\tfrac{1}{2} + it_m, \chi; \tfrac{1}{2}(a_m - 1)) \mathcal{N}_{P_R}(\tfrac{1}{2} + it_m, \chi; \tfrac{1}{2}a_m)|^2 .
\end{align}
\end{lemma}

  Our next result allows us to estimate the right-hand side expression in \eqref{Sdiffest}.  
\begin{prop}
\label{prop:RHtwisted}
Assuming the truth of GRH, we have
\begin{align*}
 \sum_{\chi \in \mathcal{T}_k} \prod_{1 \leq m\leq k} |L(\tfrac{1}{2} + it_m, \chi)|  |\mathcal{N}_{P_R}(\tfrac{1}{2} + it_m, \chi; \tfrac{1}{2}(a_m - 1)) \mathcal{N}_{P_R} ( \tfrac{1}{2} + it_m, \chi; \tfrac{1}{2}a_m)|^2  \ll   \varphi(q) \prod_{p\leq P_R} \Big(1 + \sum_{1\leq j,m \leq k} \frac{a_j a_m}{p^{1+i(t_j-t_m)}} \Big).
\end{align*}
\end{prop}

   In what follows, we shall establish Propositions \ref{prop:J1Lower}, \ref{prop:J2J1comp} and \ref{prop:RHtwisted} and then apply them to prove Proposition \ref{prop:I0Bound}.
 
\subsection{Proof of Proposition \ref{prop:J1Lower}}
\label{sect:prop5.2}
We consider the Dirichlet polynomials given by  
\begin{align*}
\begin{split}
\prod_{1 \leq m\leq k} \mathcal{N}_{j,P_R}(\tfrac{1}{2} + it_m, \chi; \tfrac{1}{2}(a_m - 1))^2 =: \sum_{n} \frac{q_j(n)\chi(n)}{n^{1/2 }}, \quad
\prod_{1 \leq m\leq k} \mathcal{N}_{P_R}(\tfrac{1}{2} + it_m, \chi; \tfrac{1}{2}(a_m - 1))^2 =: \sum_{n} \frac{q(n)\chi(n)}{n^{1/2 }}, \\
\prod_{1 \leq m\leq k} \mathcal{N}_{j,P_R}(\tfrac{1}{2} + it_m, \chi; \tfrac{1}{2}a_m)^2 =: \sum_{n} \frac{r_j(n)\chi(n)}{n^{1/2 }}, \quad \mbox{and} \quad
\prod_{1 \leq m\leq k} \mathcal{N}_{P_R}(\tfrac{1}{2} + it_m, \chi; \tfrac{1}{2}a_m)^2 =: \sum_{n} \frac{r(n)\chi(n)}{n^{1/2}}. 
\end{split}
\end{align*}
Observe that each $q_j(n)$ (resp. $r_j(n)$ ) is the two-fold Dirichlet convolution of $b_{j,P_R, {\bf t}, \tfrac{1}{2}({\bf a}-{\bf 1})}(n)$ (resp. $b_{j,P_R, {\bf t}, \tfrac{1}{2}{\bf a}}(n)$) with itself, where $\tfrac{1}{2}({\bf a}-{\bf 1})$ (resp.  $\tfrac{1}{2}\bf{a}$) denotes the vector with $j$-th component as $\tfrac{1}{2}(a_j - 1)$ (resp. $\tfrac{1}{2}a_j$).  Moreover, we set $q'_j(n)$ (resp. $r'_j(n)$) to be the two-fold Dirichlet convolution of $b'_{j,P_R, {\bf t}, \tfrac{1}{2}(\bf{a}-1)}(n)$ (resp. $b'_{j,P_R, {\bf t}, \tfrac{1}{2}\bf{a}}(n)$) with itself. We also define $q''_j(n)$ and $r''_j(n)$ in a similar fashion.  Here $q_j'$ and $r'_j$ can be regarded as multiplicative approximations of $q_j$ and $r_j$ and that $q''_j$ and $r''_j$ are non-negative multiplicative coefficients satisfying $|q_j|,|q'_j| \leq q''_j$ and $|r_j|,|r'_j| \leq r''_j$. 
The following bounds can be established in a manner similar to that of \cite[Lemma 4.5]{Curran}. Thus we omit its proof.
\begin{lemma}\label{lem:prodCoefPatrol}
 With the notation as above, we have for $p \in (P_{j-1},P_j]$ with $1 \leq j \leq R$, 
\begin{align*}
\begin{split}
q_j(p^l), r_j(p^l) \ll k^2 a_{\ast}^l l^{4k} e^{-l \log (l/2k)/4k + 2l}.
\end{split}
\end{align*}
\end{lemma}

From the above lemma, we deduce that there exists a constant $C_0$ depending on $k$ and ${\bf a}$ only, such that $q_j(p^l), r_j(p^l) \ll C_0^l$ for all $l \geq 1$. It then follows from this that, for any $n \geq 1$ and $\varepsilon>0$, 
\begin{align}
\begin{split}
\label{qrbound}
 q_j(n), r_j(n) \ll C_0^{200ka_\ast^2 K_1+100ka_\ast^2 K_2+\cdots+100ka_\ast^2 K_R} \ll q^{\varepsilon}.
\end{split}
\end{align}

We now apply Lemma \ref{lem:AFEpure} to see that for a parameter $X>0$ to be determined later,
\begin{align}
\begin{split}
\label{Jdef}
 \mathcal{J}_1 =&  \sum_{\substack{n, h, m \\ h,m \leq q^{2/10^M}}} \frac{\tau_{\bf{t}}(n) q(h)\overline{r(m)}}{n^{1/2} h^{1/2} m^{1/2}} \sum_{\chi \in X^*_q} \chi(nh)\overline \chi(m)W_{\mathfrak{a}, {\bf t}}\Big( \frac {nX}{q^{k/2}} \Big) \\
& \hspace*{1cm} +\Big( \frac {q} {\pi} \Big)^{-i\sum_{j \leq k}t_j}\sum_{\substack{n, h, m \\ h,m \leq q^{2/10^M}}} \frac{\tau_{-\bf{t}}(n) q(h)\overline{r(m)}}{n^{1/2} h^{1/2} m^{1/2}}\sum_{\chi \in X^*_q} \chi(h)\overline \chi(nm)  \Big (\frac {i^{\mathfrak{a}}q^{1/2}}{\tau(\chi)}\Big )^k W_{\mathfrak{a}, -{\bf t}} \Big(\frac {n}{Xq^{k/2}} \Big) \\
=: & \mathcal{J}_{1,1}+\mathcal{J}_{1,2}.
\end{split}
\end{align}

Note that (see \cite[Theorem 9.5]{MVa1}) for a primitive Dirichlet character $\chi$ modulo $q$,
\begin{align*}
 \tau(\chi) \overline{\tau(\chi)}=q.
\end{align*}
  Moreover, (see \cite[Theorem 9.5]{MVa1})
\begin{align*}
 \overline{\tau(\chi)} =\chi(-1)\tau(\overline\chi).
\end{align*}

Using the above formulas, we get
\begin{align}
\begin{split}
\label{sumchi}
 &\sum_{\chi \in X^*_q} \chi(h)\overline \chi(nm)  \Big (\frac {i^{\mathfrak{a}}q^{1/2}}{\tau(\chi)}\Big )^k W_{\mathfrak{a}, -{\bf t}} \Big(\frac {n}{Xq^{k/2}} \Big) 
 = \sum_{\chi \in X^*_q} \chi(h)\chi(-1)^k\overline \chi(nm) \tau(\overline\chi)^k \Big (\frac {i^{\mathfrak{a}}q^{1/2}}{q} \Big )^k W_{\mathfrak{a}, -{\bf t}} \Big(\frac {n}{Xq^{k/2}} \Big) \\
 =& \big (iq^{-1/2}\big )^kW_{1, -{\bf t}}\Big(\frac {n}{Xq^{k/2}}\Big)\sum_{\substack{\chi \in X^*_q \\ \chi(-1)=1}} \chi(h)\overline \chi(nm) \tau(\overline\chi)^k  +\big (iq^{-1/2}\big )^kW_{-1, -{\bf t}} \Big( \frac {n}{Xq^{k/2}} \Big)\sum_{\substack{\chi \in X^*_q \\ \chi(-1)=-1}} \chi(h)\overline \chi(nm) \tau(\overline\chi)^k \\
 =& \big (iq^{-1/2}\big )^kW_{1, -{\bf t}} \Big(\frac {n}{Xq^{k/2}}\Big)\sum_{\substack{\chi \in X^*_q }} \frac {\chi(1)+\chi(-1)}{2}\chi(h)\overline \chi(nm) \tau(\overline\chi)^k  \\
 & \hspace*{2cm} +\big (iq^{-1/2}\big )^kW_{-1, -{\bf t}}\Big(\frac {n}{Xq^{k/2}}\Big)\sum_{\substack{\chi \in X^*_q }} \frac {\chi(1)-\chi(-1)}{2}\chi(h)\overline \chi(nm) \tau(\overline\chi)^k \\
 =& \frac 12\big (iq^{-1/2}\big )^k\Big (W_{1, -{\bf t}} \Big(\frac {n}{Xq^{k/2}}\Big)+W_{-1, -{\bf t}}\Big(\frac {n}{Xq^{k/2}}\Big)\Big )\sum_{\substack{\chi \in X^*_q }} \chi(h)\overline \chi(nm) \tau(\overline\chi)^k  \\
 & \hspace*{2cm} +\frac 12\big (iq^{-1/2}\big )^k\Big (W_{1, -{\bf t}} \Big(\frac {n}{Xq^{k/2}}\Big)-W_{-1, -{\bf t}}\Big(\frac {n}{Xq^{k/2}}\Big)\Big )\sum_{\substack{\chi \in X^*_q }} \chi(-h)\overline \chi(nm) \tau(\overline\chi)^k.
\end{split}
\end{align}

  We write $\mu$ for the M\"obius function and note the following relation (see \cite[Lemma 1]{Sound2007}).  For $(a, q)=1$,
\begin{align}
\label{sumchistar}
 & \sumstar_{\substack{ \chi \shortmod q }}\chi(a)=\sum_{c | (q, a-1)}\mu(q/c)\varphi(c),
\end{align}
  where $\sum^*$ stands for the sum over primitive characters modulo $q$. \newline

  In particular, setting $a=1$ in \eqref{sumchistar} implies that
\begin{align}
\label{chistar}
 & \phis(q)=\sum_{c | q}\mu(q/c)\varphi(c).
\end{align}

   For any integer $v$ that is co-prime to $q$, we write $\overline{v}$ for $v^{-1} \pmod q$ and the multiple Kloosterman's sum $S_k(v,q)$ is defined by (see \cite[(11.55)]{iwakow})
\begin{align*}
  S_k(v, q)=\sumstar_{\substack{ x_1, \cdots, x_k \shortmod q \\ x_1x_2\cdots x_k=v}} \exp \Big( \frac {x_1+\cdots+x_k}{q} \Big).
\end{align*}
The well-known Weil's bound for Kloosterman's sum (see \cite[p. 309]{iwakow}) asserts that
\begin{align}
\label{Weilbound}
  |S_k(v, q)| \leq d(q)q^{(k-1)/2},
\end{align}
  where $d(\cdot)$ denotes the divisor function. \newline

  We deduce from \eqref{chistar} and \eqref{Weilbound} that
\begin{align*}
 \sum_{\substack{\chi \in X^*_q }} \chi(\pm h)\overline \chi(nm) \tau(\overline\chi)^k =&  \sum_{\substack{\chi \in X^*_q }} \chi(\pm h)\overline \chi(nm)   \sumstar_{\substack{ v \shortmod q }} \overline \chi(\overline{v}) S_k(\overline{v},q)=\sum_{c | q}\mu(q/c)\varphi(c)\sumstar_{\substack{ v \shortmod q \\ vh \equiv \pm nm \shortmod c} } S_k(\overline{v},q) \\
 \ll & d(q)q^{(k-1)/2}\sum_{c | q}\varphi(c)\sumstar_{\substack{ v \shortmod q \\ vh \equiv \pm nm \shortmod c} }1.
\end{align*}

The number of $v \pmod q$ such that $vh \equiv \pm mn \pmod c$ is easily seen to be
$\leq q(h,c)/c$. These estimates, together with the bound $d(n) \ll n^{\varepsilon}$ for any $\varepsilon>0$ (see \cite[(2.20)]{MVa1}), lead to
\begin{align}
\label{sumchistarest}
 \sum_{\substack{\chi \in X^*_q }} \chi(\pm h)\overline \chi(nm) \tau(\overline\chi)^k \leq d(q)q^{(k-1)/2}\sum_{c | q}\varphi(c)\frac {q(h,c)}{c} \ll hq^{(k+1)/2+\varepsilon}.
\end{align}

Now, from \eqref{qrbound}--\eqref{sumchistar} and \eqref{sumchistarest},
\begin{align*}
 \mathcal{J}_{1,2} \ll &\Big( \frac {q} {\pi} \Big)^{-i\sum_{j \leq k}t_j}\sum_{\substack{n, h, m \\ h,m \leq q^{2/10^M}}} \frac{\tau_{-\bf{t}}(n) q(h)\overline{r(m)}}{n^{1/2} h^{1/2} m^{1/2}}\sum_{\chi \in X^*_q} \chi(h)\overline \chi(nm)  \Big (\frac {i^{\mathfrak{a}}q^{1/2}}{\tau(\chi)}\Big )^k W_{\mathfrak{a}, -{\bf t}} \Big(\frac {n}{Xq^{k/2}} \Big) \\
 \ll & q^{1/2+\varepsilon}\sum_{\substack{n, h, m \\ h,m \leq q^{2/10^M}}} \frac{|\tau_{-\bf{t}}(n)|h}{n^{1/2} h^{1/2} m^{1/2}}\Big ( \Big | W_{-1, -{\bf t}} \Big(\frac {n}{Xq^{k/2}} \Big) \Big |+\Big | W_{-1, -{\bf t}} \Big(\frac {n}{Xq^{k/2}} \Big) \Big | \Big )  \ll q^{4/10^{M}}q^{1/2+\varepsilon}X^{1/2+\varepsilon}q^{k/4}.
\end{align*}

Moreover,
\begin{align}
\begin{split}
\label{sumchi1}
 \sum_{\chi \in X^*_q} & \chi(nh)\overline \chi(m)W_{\mathfrak{a}, {\bf t}}\Big(\frac {nX}{q^{k/2}}\Big)  = W_{1, {\bf t}}\Big(\frac {nX}{q^{k/2}}\Big)\sum_{\substack{\chi \in X^*_q \\ \chi(-1)=1}} \chi(nh)\overline \chi(m)  +W_{-1, {\bf t}} \Big(\frac {nX}{q^{k/2}} \Big)\sum_{\substack{\chi \in X^*_q \\ \chi(-1)=-1}} \chi(nh)\overline \chi(m)  \\
 =& W_{1, {\bf t}}\Big(\frac {nX}{q^{k/2}}\Big)\sum_{\substack{\chi \in X^*_q }} \frac {\chi(1)+\chi(-1)}{2}\chi(nh)\overline \chi(m) +W_{-1, {\bf t}}\Big(\frac {nX}{q^{k/2}}\Big)\sum_{\substack{\chi \in X^*_q }} \frac {\chi(1)-\chi(-1)}{2}\chi(nh)\overline \chi(m)  \\
 =& \frac 12\Big (W_{1, {\bf t}}\Big(\frac {nX}{q^{k/2}}\Big)+W_{-1, {\bf t}}\Big(\frac {nX}{q^{k/2}}\Big)\Big )\sum_{\substack{\chi \in X^*_q }} \chi(nh)\overline \chi(m)  +\frac 12\Big (W_{1, {\bf t}}\Big(\frac {nX}{q^{k/2}}\Big)-W_{-1, {\bf t}}\Big(\frac {nX}{q^{k/2}}\Big)\Big )\sum_{\substack{\chi \in X^*_q }} \chi(-nh)\overline \chi(m).
\end{split}
\end{align}

  We apply \eqref{Jdef}, \eqref{sumchistar} and \eqref{sumchi1} to arrive at
\begin{align}
\begin{split}
\label{J11decomp}
 \mathcal{J}_{1,1} =& \frac 12\sum_{c | q}\mu(q/c)\varphi(c) \sum_{\substack{n, h, m \\ h,m \leq q^{2/10^M}\\ nh \equiv m \bmod c}} \frac{\tau_{\bf{t}}(n) q(h)\overline{r(m)}}{n^{1/2} h^{1/2} m^{1/2}} \Big (W_{1, {\bf t}}\Big(\frac {nX}{q^{k/2}}\Big)+W_{-1, {\bf t}}\Big(\frac {nX}{q^{k/2}} \Big)\Big ) \\
 & \hspace*{1cm} +\frac 12\sum_{c | q}\mu(q/c)\varphi(c) \sum_{\substack{n, h, m \\ h,m \leq q^{2/10^M}\\ nh \equiv -m \bmod c}} \frac{\tau_{\bf{t}}(n) q(h)\overline{r(m)}}{n^{1/2} h^{1/2} m^{1/2}} \Big (W_{1, {\bf t}}\Big(\frac {nX}{q^{k/2}}\Big)-W_{-1, {\bf t}}\Big(\frac {nX}{q^{k/2}}\Big)\Big ).
\end{split}
\end{align}

  We recast the condition $nh \equiv \pm m \pmod c$ as $nh = \mp m+cl$.  We first consider the contribution from the terms with $l \geq 1$. As the treatments as similar, we only dead with the case $nh + m=cl$ here. By the rapid decay of $W_{\pm 1, {\bf t}}(x)$ given in \eqref{Westimation}, we may assume that $n \leq (q^{k/2}/X)^{1+\varepsilon}$ and the condition $nh + m=cl$ then implies that $l \leq  q^{k/2+2/10^{M}+\varepsilon}/(Xc)$.
We consider the sum over the terms satisfying $nh \geq m$. This implies that $nh \geq cl/2$ so that, together with the observation that $q(h), h(m) \ll q^{\varepsilon}$, the total contribution from these terms is
\begin{align}
\label{Firstmomentsum1nondiag1}
 \ll &  \sum_{c | q} \varphi(c)  q^{\varepsilon}X^{\varepsilon} \sum_{m  \leq q^{2/10^{M}}}  \sum_{l \leq q^{k/2+2/10^{M}+\varepsilon}/(Xc)}\frac {1}{\sqrt{mlc}}
\ll X^{-1/2+\varepsilon}q^{k/4+2/10^{M}+\varepsilon}.
\end{align}

  Similarly, the contribution from the terms $m \geq nh$ and $nh + m=cl$ with $l \geq 1$ (by noting that
$l \leq  q^{2/10^{M}}/c$ and $m \geq lc$) is
\begin{align}
\label{Firstmomentsum1nondiag2}
 \ll &  \sum_{c | q} \varphi(c)  q^{\varepsilon}X^{\varepsilon} \sum_{h  \leq q^{2/10^{M}}}\sum_{n \leq (q^{k/2}/X)^{1+\varepsilon}}
  \sum_{l \leq q^{2/10^{M}}/c}\frac {1}{\sqrt{nhlc}}
\ll X^{-1/2+\varepsilon}q^{k/4+2/10^{M}+\varepsilon}.
\end{align}

Setting $X=q^{-1/2-2/10^{M}}$, from \eqref{Jdef}, \eqref{chistar}, \eqref{J11decomp},  \eqref{Firstmomentsum1nondiag1} and \eqref{Firstmomentsum1nondiag2}, upon taking $\delta$ sufficiently small (so that $M$ is large enough), we have, for $k=2$, 
\begin{align}
\label{J11estmation}
\mathcal{J}_{1} \gg & \phis(q)  \sum_{\substack{n, h, m  \\ h,m  \leq q^{2/10^M}\\ nh = m }} \frac{\tau_{\bf{t}}(n) q(h)\overline{r(m)}}{n^{1/2} h^{1/2} m^{1/2}} \Big (W_{1, {\bf t}}\Big(\frac {nX}{q^{k/2}}\Big)+W_{-1, {\bf t}}\Big(\frac {nX}{q^{k/2}}\Big)\Big )+O(q^{1-\varepsilon}).
\end{align}
   
The condition $nh = m$ implies that $n \leq m \leq q^{2/10^M}$.  Hence it follows from \eqref{Westimation} and \eqref{J11estmation} that 
\[
\mathcal{J}_1 \gg \phis(q)  \sum_{\substack{n, h, m\\ nh = m }} \frac{\tau_{\bf{t}}(n) q(h)\overline{r(m)}}{n^{1/2} h^{1/2} m^{1/2}} +O(q^{1-\varepsilon}).
\]
  
 We recast the summation above by multiplicativity as
\[
\sum_{\substack{h,m \\ h|m}} \frac{\tau_{\bf{t}}(m/h) q(h) \overline{r(m)}}{m} = \prod_{1 \leq  j\leq R} \sum_{\substack{p| h,m \Rightarrow p\in (P_{j-1},P_j] \\ h|m }} \frac{\tau_{\bf{t}}(m/h) q_j(h) \overline{r_j(m)}}{m}.
\]

If $q_j(h) \neq q'_j(h)$, then $\Omega(h) \geq 100 a_\ast^2 K_j$. Thus, we may apply Rankin's trick to replace $q_j,r_j$ with $q'_j, r'_j$  respectively for each $j$ at the cost of an error of size
\begin{align*}
\ll &e^{-100a_\ast^2 K_j} \sum_{\substack{p| h,m \Rightarrow p\in (P_{j-1},P_j] \\ h|m }} \frac{e^{\Omega(h)}|\tau_{\bf{t}}(m/h)| q''_j(h) r''_j(m)}{m} \ll e^{-100a_\ast^2 K_j} \prod_{p\in(P_{j-1},P_j]} \left(1 + \frac{2 e a_\ast^2}{p} + O \left(\frac{1}{p^2}\right)\right) \ll  e^{-50 a_\ast^2 K_j}.
\end{align*}
 Here the $O \left(p^{-2}\right)$ term above emerges from the estimate $|\tau_{{\bf t}}(p^l)| \ll_{\varepsilon} p^{\varepsilon l}$ for any integer $l \geq 1$ and any real $\varepsilon>0$ (which can be established in a manner similar to \cite[(2.20)]{MVa1}) and Lemma \ref{lem:prodCoefPatrol}.
It follows that
\begin{align*}
\mathcal{J}_1 \gg & \phis(q)\prod_{1 \leq  j\leq R}\left( \prod_{p\in(P_{j-1},P_j]}\left( \sum_{0 \leq x \leq y} \frac{\tau_{{\bf t}}(p^{y-x}) q_j(p^x) \overline{r_j(p^y)}}{p^y} \right) + O (e^{-50 a_\ast^2 K_j})\right)+O(q^{1-\varepsilon})\\
=&  \phis(q)\ \prod_{1 \leq  j\leq R}\left( \prod_{p\in(P_{j-1},P_j]}\left( 1 + \frac{q_j(p) \overline{r_j(p)} + \tau_{{\bf t}}(p) \overline{r_j(p)} }{p} + O \left(\frac{1}{p^2}\right)\right) + O (e^{-50 a_\ast^2 K_j})\right)+O(q^{1-\varepsilon}) \\
=&  \phis(q)\ \prod_{1 \leq  j\leq R}\left( \prod_{p\in(P_{j-1},P_j]}\left( 1 + \frac{1}{p} \Bigg|a_{P_R}(p) \sum_{1 \leq  m \leq k} a_m p^{-it_m}\Bigg|^2 + O \left(\frac{1}{p^2}\right)\right) + O (e^{-50 a_\ast^2 K_j})\right)+O(q^{1-\varepsilon}).
\end{align*}
The last equality above follows by observing that
\[
q_j(p) + \tau_{{\bf t}}(p) = a_{P_R}(p) \sum_{1 \leq m\leq k} a_m p^{-i t_m} = r_j(p).
\]

 By further recalling that  $a_{P_R}(p) = p^{-1/\log P_R} (1 - \log p/ \log P_R)$, we see that
\begin{align}
\begin{split}
\label{J1lowerbound}
 \mathcal{J}_1 \gg &  \phis(q)\prod_{1 \leq  j\leq R}\left( \prod_{p\in(P_{j-1},P_j]}\left( 1 + \frac{1}{p} \Bigg|\sum_{1 \leq  m \leq k} a_m p^{-it_m}\Bigg|^2 + O \left( \frac{\log p}{p \log P_R} + \frac{1}{p^2}\right)\right) + O (e^{-50 a_\ast^2 K_j})\right)+O(q^{1-\varepsilon}) \\
\gg &  \varphi(q)\prod_{1 \leq  j\leq R}\left( \prod_{p\in(P_{j-1},P_j]}\left( 1 + \frac{1}{p} \Bigg|\sum_{1 \leq  m \leq k} a_m p^{-it_m}\Bigg|^2 \right) \right),
\end{split}
\end{align}
  where the last minorant above follows from Lemma \ref{RS}, the observation that $\prod_{1 \leq j\leq R} (1 +O(e^{-50a_\ast^2 K_j}))^2 = O(1)$ and our assumptions $\phis(q) \geq \eta \varphi(q)$ and $\varphi(q) \gg q^{1-\varepsilon}$ (which follows from \cite[Theorem 2.9]{MVa1}). This completes the proof of the proposition. 

\subsection{Proof of Proposition \ref{prop:J2J1comp}}

  In order to analyze $\mathcal{J}_2$, we want to first partition the set $X^*_q \setminus \mathcal{T}_k$ into unions of certain subsets. To this end, we 
define for any subset $\mathcal{A}$ of $[k] =: \{1,\ldots, k\}$, 
\[
\mathcal{T}_\mathcal{A} =:  \left\{ \chi \in X^*_q: |P_{j,P_R} (\tfrac{1}{2} + i t_m, \chi)| \leq K_j \text{ for all } 1\leq j\leq R \text{ if and only if } m \in \mathcal{A}\right\},
\]
  With this notation, we see that 
\[
X^*_q \setminus \mathcal{T}_k =  \bigsqcup_{\mathcal{A} \subsetneq [k]} \mathcal{T}_\mathcal{A} .
\]

 Now, to further decompose each $\mathcal{A}$, we define for $1 \leq j \leq R$, 
\begin{align*}
\mathcal{B}_j =: & \ \{1 \leq m \leq k: |\mathcal{P}_{r,P_s} (\tfrac{1}{2} + i t_m, \chi)| \leq K_j \text{ for all } 1\leq r < j\text{ and } r\leq s \leq  R \\
& \hspace*{2cm} \text{ but } |\mathcal{P}_{j,P_s} (\tfrac{1}{2} + i t_m, \chi)|  > K_{j} \text{ for some } j \leq s \leq R \}.
\end{align*}
We then write $\mathcal{A} = [k]\setminus [a]$ to note that for each $\chi \in \mathcal{T}_\mathcal{A}$, there is a function $F_{\chi}: [a] \rightarrow [R]$ such that $t_j \in \mathcal{B}_{F(j)}$. It follows that we have
\[
\mathcal{T}_\mathcal{A} =   \bigsqcup_{A \subsetneq [k]} \mathcal{B}_{\mathcal{A},n}, \quad \mbox{where} \quad \mathcal{B}_{\mathcal{A},n} = \{\chi \in \mathcal{T}_\mathcal{A}: \min_{j \in [a]} F_{\chi}(j) = n \}.
\]
We thus conclude that
\[
X^*_q \setminus \mathcal{T}_k =  \bigsqcup_{n\leq R} \bigsqcup_{\mathcal{A}\subsetneq [k]} \mathcal{B}_{\mathcal{A},n}.
\]

Suppose now that $n > 1$. We apply Lemma \ref{lem:logZetaUpperBound} by setting $k=a_1=1, x = P_{n-1}$ there.  This gives that for any $\chi \in X^*_q$ and $1 \leq m \leq k$, 
\begin{align*}
|L(\tfrac{1}{2} + it_m, \chi)| 
\ll &
\exp\Big(\Re \Big( \sum_{1 \leq j < n}  \mathcal{P}_{j,P_{n-1}} (\tfrac{1}{2}  + it_m, \chi) \Big)+ \frac {A+1}{ c_{n - 1}} \Big).
\end{align*}

 It follows from this, the definition of $\mathcal{B}_{\mathcal{A},n}$ and Lemma \ref{lem:expTaylorSeries} that 
\begin{align}
\begin{split}
\label{sumB}
\sum_{\mathcal{B}_{\mathcal{A},n}} & \prod_{1 \leq m\leq k} |L(\tfrac{1}{2} + it_m, \chi)| |\mathcal{N}_{P_R}(\tfrac{1}{2} + it_m, \chi; \tfrac{1}{2}(a_m - 1))\mathcal{N}_{P_R}(\tfrac{1}{2} + it_m, \chi; \tfrac{1}{2}a_m)|^2  \\
\ll & ~e^{\frac {A+1}{c_{n - 1}}}  \sum_{\mathcal{B}_{\mathcal{A},n}}\prod_{m = 1}^k \prod_{1 \leq j < n} \exp\left(\Re\mathcal{P}_{j,P_{n-1}} (\tfrac{1}{2}  + it_m, \chi) \right) |\mathcal{N}_{P_R}(\tfrac{1}{2} + it_m, \chi; \tfrac{1}{2}(a_m - 1))\mathcal{N}_{P_R}(\tfrac{1}{2} + it_m, \chi; \tfrac{1}{2}a_m)|^2   \\
\ll&  ~e^{\frac {A+1}{c_{n - 1}}} \max_{\substack{\ell \in [a] \\ s \in [R]}}
\sum_{\chi \in X_q} |\mathcal{P}_{n,P_s}(\tfrac{1}{2} + i t_{\ell}, \chi)/K_{n}|^{2 \lceil 1/ 20 c_{n}\rceil} \prod_{m = 1}^k \prod_{1 \leq j < n} |\mathcal{N}_{j,P_{n-1}} (\tfrac{1}{2}+ it_m, \chi ; \tfrac{1}{2}) |^2  \\
&\qquad\qquad\qquad\qquad\times |\mathcal{N}_{P_R}(\tfrac{1}{2} + it_m, \chi; \tfrac{1}{2}(a_m - 1))\mathcal{N}_{P_R}(\tfrac{1}{2} + it_m, \chi; \tfrac{1}{2}a_m)|^2.
\end{split}
\end{align}
  Here $\lceil \ell \rceil = \min \{ m \in \intz : \ell \leq m \}$ is the ceiling of $\ell$. \newline

  To evaluate the last sum above,  we consider, for a given $j$, sums of the form:
$$A_j(\chi)= \sum_{n \in S_j}a_j(n)\chi(n),$$
 where $\chi$ is a Dirichlet character modulo $q$. If the elements of $S_{i}$ and $S_{j}$ are mutually co-prime whenever $i \neq j$ and that $\prod^J_{j=1}n_j =o(q)$ for all $n_j \in S_j$, then an application of the  orthogonality relation \eqref{orthrel} implies that 
\begin{align}
\begin{split}
\label{orthorelprod}
 & \sum_{\chi \in X_q} \prod^J_{j=1}|A_j(\chi)|^2=\varphi(q)\prod^J_{j=1} \big ( \frac 1{\varphi(q)} \sum_{\chi \in X_q} |A_j(\chi)|^2 \big ). 
\end{split}
\end{align}   
    We note further that by expanding the last product in \eqref{sumB} into Dirichlet series, the length of which is $\ll q^{1-\varepsilon}$ by arguing similar to our discussions on the length of $\prod_{1 \leq m\leq k} \mathcal{N}_{x}(s, \chi; a_m)$ in Section \ref{sect:J}. It follows from this and \eqref{orthorelprod} that 
\begin{align}
\begin{split}
\label{sumBdecomp}
  \sum_{\chi \in X_q} & |\frac {\mathcal{P}_{n,P_s}(\tfrac{1}{2} + i t_{\ell}, \chi)}{K_{n}}|^{2 \lceil \frac {1}{20 c_{n}}\rceil} \prod_{m = 1}^k \prod_{1 \leq j < n} |\mathcal{N}_{j,P_{n-1}} (\tfrac{1}{2}+ it_m, \chi ; \tfrac{1}{2}) |^2 |\mathcal{N}_{P_R}(\tfrac{1}{2} + it_m, \chi; \tfrac{1}{2}(a_m - 1))\mathcal{N}_{P_R}(\tfrac{1}{2} + it_m, \chi; \tfrac{1}{2}a_m)|^2 \\
=& \varphi(q)\prod_{1 \leq j<n} \big (\frac 1{\varphi(q)} \sum_{\chi \in X_q} \prod_{k = 1}^m|\mathcal{N}_{j,P_{n-1}} (\tfrac{1}{2}+ it_m, \chi ; \tfrac{1}{2}) \mathcal{N}_{j,P_R}(\tfrac{1}{2} + it_m, \chi; \tfrac{1}{2}(a_m - 1))\mathcal{N}_{j,P_R}(\tfrac{1}{2} + it_m, \chi; \tfrac{1}{2}a_m)|^2 \big )\\
& \hspace*{1cm} \times \frac 1{\varphi(q)}\sum_{\chi \in X_q} |\mathcal{P}_{n,P_s}(\tfrac{1}{2} + i t_{\ell}, \chi)/K_{n}|^{2 \lceil 1/ 20 c_{n}\rceil}  \prod_{m = 1}^k |\mathcal{N}_{n,P_R}(\tfrac{1}{2} + it_m, \chi; \tfrac{1}{2}(a_m - 1))\mathcal{N}_{n,P_R}(\tfrac{1}{2} + it_m, \chi; \tfrac{1}{2}a_m)|^2 \\
& \hspace*{1cm} \times \prod_{n < j \leq R}\frac 1{\varphi(q)}\sum_{\chi \in X_q} \prod_{m = 1}^k |\mathcal{N}_{j,P_R}(\tfrac{1}{2} + it_m, \chi; \tfrac{1}{2}(a_m - 1))\mathcal{N}_{j,P_R}(\tfrac{1}{2} + it_m, \chi; \tfrac{1}{2}a_m)|^2. 
\end{split}
\end{align}  
  
  It therefore remains to estimate the sums appearing in the last expression above. This done through three propositions. As the arguments are similar to those given in Section \ref{sect:J} and Section \ref{sect:prop5.2}, we shall only sketch their proofs.
\begin{prop}
\label{prop:RHMVDP}
We have, for $1 \leq j < n$, 
\begin{align*}
\sum_{\chi \in X_q} \prod_{m = 1}^k & |\mathcal{N}_{j,P_{n-1}} (\tfrac{1}{2}+ it_m, \chi ; \tfrac{1}{2}) \mathcal{N}_{j,P_R}(\tfrac{1}{2} + it_m, \chi; \tfrac{1}{2}(a_m - 1))\mathcal{N}_{j,P_R}(\tfrac{1}{2} + it_m, \chi; \tfrac{1}{2}a_m)|^2 \\
\ll  & \varphi(q) \prod_{p \in (P_{j-1},P_j]}  \left(1 + \frac{1}{p} \sum_{1\leq j , m \leq k} \frac{a_j a_m}{p^{i(t_j - t_m)}}  + O  \left(\frac{\log p}{p \log P_{n-1}} + \frac{1}{p^2}\right) \right) + O(e^{-50a_\ast^2 K_j}).
\end{align*}
\end{prop}
\begin{proof}
  We write $ \bf{\tfrac12}$ for the $m$-dimensional vector whose coordinates are all $1/2$ and 
\[
\label{triplecon}
\prod_{k = 1}^m\mathcal{N}_{j,P_{n-1}} (\tfrac{1}{2}+ it_m, \chi ; \tfrac{1}{2}) \mathcal{N}_{j,P_R}(\tfrac{1}{2} + it_m, \chi; \tfrac{1}{2}(a_m - 1))\mathcal{N}_{j,P_R}(\tfrac{1}{2} + it_m, \chi; \tfrac{1}{2}a_m)=\sum_n \frac {a(n)}{n^{1/2}},
\]
 where $a(n)$ is given by the triple convolution
\[
a(n) =: b_{j,P_{n-1}, \bf{t}, \bf{\tfrac 12}}\ast b_{j,P_R, \bf{t}, \tfrac{1}{2}(\bf{a} - 1)} \ast b_{j,P_R, \bf{t}, \tfrac{1}{2}\bf{a}} (n).
\]

Now Rankin's trick renders that, at a cost of an error of size $O(e^{-50a_\ast^2 K_j})$, one can replace $a(n)$ in \eqref{triplecon} with the multiplicative coefficient $a'(n)$ given by
\[
a'(n) =: b'_{j,P_{n-1}, \bf{t}, \bf{\tfrac12}}\ast b'_{j,P_R, \bf{t}, \tfrac{1}{2}(\bf{a} - 1)} \ast b'_{j,P_R, \bf{t}, \tfrac{1}{2}\bf{a}} (n).
\]

By Lemma \ref{lem:nCoeffPatrol}, one may then show that
\[
|a'(p)|^2 = \sum_{1\leq j , m \leq k} \frac{a_j a_m}{p^{i(t_j - t_m)}}  + O  \left(\frac{\log p}{\log P_{n-1}}\right), \quad
\sum_{r \geq 2} \frac{|a(p^r)|^2}{p^r} =O( \frac{1}{p^2}).
\]
The proposition follows from multiplicativity and the orthogonality relation \eqref{orthrel}.
\end{proof}

\begin{prop}
\label{prop:RHMVDP1}
We have, for $1 < n \leq R$, 
\begin{align}
\begin{split}
\label{sumn}
\sum_{\chi \in X_q} |\mathcal{P}_{n,P_s}(\tfrac{1}{2} + i t_{\ell}, \chi)/K_{n}|^{2 \lceil 1/ 20 c_{n}\rceil} & \prod_{m = 1}^k |\mathcal{N}_{n,P_R}(\tfrac{1}{2} + it_m, \chi; \tfrac{1}{2}(a_m - 1))\mathcal{N}_{n,P_R}(\tfrac{1}{2} + it_m, \chi; \tfrac{1}{2}a_m)|^2 \\
\ll \varphi(q) e^{-\log(1/c_n)/40c_n} .
\end{split}
\end{align}
\end{prop}
\begin{proof}
Cauchy's inequality renders that the left-hand side of \eqref{sumn} is 
\begin{align}
\begin{split}
\label{sumnestimation}
\leq \left(\sum_{\chi \in X_q} |\mathcal{P}_{n,P_s}(\tfrac{1}{2} + i t_{\ell}, \chi)/K_{n}|^{2 \lceil 1/ 20 c_{n}\rceil}  \right)^{1/2} \left(\sum_{\chi \in X_q}  \prod_{m = 1}^k |\mathcal{N}_{P_R}(\tfrac{1}{2} + it_m, \chi; \tfrac{1}{2}(a_m - 1))\mathcal{N}_{P_R}(\tfrac{1}{2} + it_m, \chi; \tfrac{1}{2}a_m)|^2 \right)^{1/2}.
\end{split}
\end{align}
 Similar to the proof of \cite[Lemma 3]{Sound2009}, we can show that for any $x \in \rear$ with $2 \leq x \leq q$ and $k\in \intz$ satisfying $x^k \leq q/\log q$, and any complex numbers $a(p)$, 
\begin{align}
\begin{split}
\label{powerest}
\sum_{\chi \in X_q} |\sum_{p \leq x}\frac {a(p)\chi(p)}{p^{1/2}}|^{2k} & \ll \varphi(q)k! \Big( \sum_{p \leq x}\frac {|a(p)|^2}{p} \Big)^k.
\end{split}
\end{align}
Using the above and arguing as in the proof of \cite[Proposition 4.2]{Curran}, the first sum in \eqref{sumnestimation} is
\begin{align}
\begin{split}
\label{firstsumest}
 \ll \varphi(q) e^{-\log(1/c_n)/20c_n}.
\end{split}
\end{align}  
We estimate the second sum in \eqref{sumnestimation} by arguing in a fashion similar to the proof of Proposition \ref{prop:RHMVDP} while applying \eqref{1xbound}.  The sum in question is thus 
\begin{align}
\begin{split}
\label{secondsumest}
 \ll \varphi(q)\prod_{p\in {(P_{n-1},P_n]}} \left(1 + O \left(\frac{1}{p}\right)\right) \ll \exp\left(\sum_{p\in {(P_{n-1},P_n]}} O \left(\frac{1}{p}\right)\right) \ll \varphi(q), 
\end{split}
\end{align}
 where the last majorant above follows from Lemma \ref{RS} by noting that $\log P_n/\log P_{n-1} = e$ for $n > 1$. We deduce readily \eqref{sumn} from \eqref{sumnestimation}--\eqref{secondsumest}, completing the proof of the proposition.
\end{proof}

  We shall omit the proof of the next proposition as it is similar to the proof of Proposition \ref{prop:RHMVDP}.
\begin{prop}
\label{prop:RHMVDP2}
We have, for $n< j \leq R$,
\begin{align*}
\sum_{\chi \in X_q} \prod_{m = 1}^k |\mathcal{N}_{j,P_R}&(\tfrac{1}{2} + it_m, \chi; \tfrac{1}{2}(a_m - 1))\mathcal{N}_{j,P_R}(\tfrac{1}{2} + it_m, \chi; \tfrac{1}{2}a_m)|^2  \leq \varphi(q) \prod_{p \in (P_{j-1},P_j]}  \left(1 + O \left(\frac{1}{p}\right) \right) + O(e^{-50a_\ast^2 K_j}).
\end{align*}
\end{prop}

  We now apply Propositions \ref{prop:RHMVDP}--\ref{prop:RHMVDP2} to estimate the sums in \eqref{sumBdecomp}. Together with \eqref{sumB}, we see that
\begin{align}
\begin{split}
\label{Bsumest}
&\sum_{\mathcal{B}_{\mathcal{A},n}}  \prod_{1 \leq m\leq k} |L(\tfrac{1}{2} + it_m, \chi)| |\mathcal{N}_{P_R}(\tfrac{1}{2} + it_m, \chi; \tfrac{1}{2}(a_m - 1))\mathcal{N}_{P_R}(\tfrac{1}{2} + it_m, \chi; \tfrac{1}{2}a_m)|^2  \\
\ll & \varphi(q) \prod_{1 \leq j < n} \left(\prod_{p \in (P_{j-1},P_j]}  \left(1 + \frac{1}{p} \sum_{1\leq j , m \leq k} \frac{a_j a_m}{p^{i(t_j - t_m)}}  + O  \left(\frac{\log p}{p \log P_{n-1}} + \frac{1}{p^2}\right) \right) + O(e^{-50a_\ast^2 K_j}) \right)\\
&\qquad\qquad\times e^{A/c_{n-1} - \log(1/c_n)/40c_n} \prod_{n < j \leq R} \left(\prod_{p \in (P_{j-1},P_j]}  \left(1 + O \left(\frac{1}{p}\right) \right) + O(e^{-50a_\ast^2 K_j})\right) \\
= & \varphi(q) \prod_{1 \leq j < n} \left(\prod_{p \in (P_{j-1},P_j]}  \left(1 + \frac{1}{p} \Big|\sum_{1 \leq  m \leq k} a_m p^{-it_m}\Big|^2  + O  \left(\frac{\log p}{p \log P_{n-1}} + \frac{1}{p^2}\right) \right) + O(e^{-50a_\ast^2 K_j}) \right)   \\
&\qquad\qquad\times e^{A/c_{n-1} - \log(1/c_n)/40c_n} \prod_{n < j \leq R} \left(\prod_{p \in (P_{j-1},P_j]}  \left(1 + O \left(\frac{1}{p}\right) \right) + O(e^{-50a_\ast^2 K_j})\right) \\
\ll & \varphi(q)e^{A/c_{n-1} - \log(1/c_n)/40c_n} \prod_{1 \leq j \leq R}(1+O(e^{-50a_\ast^2 K_j}))\prod_{p \leq P_{n-1}}\Big(1+O  \Big(\frac{\log p}{p \log P_{n-1}} + \frac{1}{p^2} \Big) \Big)  
\prod_{P_{n} \leq p \leq P_{R}}\Big(1+O  \Big(\frac{1}{p} \Big) \Big) \\
& \qquad\qquad\times  \prod_{1 \leq j < n} \left(\prod_{p \in (P_{j-1},P_j]}  \left(1 + \frac{1}{p} \Big|\sum_{1 \leq  m \leq k} a_m p^{-it_m}\Big|^2 \right)\right) \\
\ll & \varphi(q)e^{A/c_{n-1} - \log(1/c_n)/40c_n} \exp \Big(\sum_{p \leq P_{n-1}}O  \Big(\frac{\log p}{p \log P_{n-1}} + \frac{1}{p^2}\Big)+ 
\sum_{P_{n} \leq p \leq P_{R}}O  \Big(\frac{1}{p}\Big) \Big) \\
& \qquad\qquad\times  \prod_{1 \leq j < n} \left(\prod_{p \in (P_{j-1},P_j]}  \left(1 + \frac{1}{p} \Big|\sum_{1 \leq  m \leq k} a_m p^{-it_m}\Big|^2 \right)\right).
\end{split}
\end{align}
The last bound above follows from the observation $\prod_{1 \leq j\leq R} (1 +O(e^{-50a_\ast^2 K_j}))^2 = O(1)$ and from the bound $1 + x \leq e^x$ for all real $x > 0$.  Now Lemma \ref{RS} gives that
\begin{align}
\label{expest}
 \exp \Big(\sum_{p \leq P_{n-1}}O \Big(\frac{\log p}{p \log P_{n-1}} + \frac{1}{p^2} \Big) \Big) \ll 1 \quad \mbox{and} \quad \sum_{P_{n} \leq p \leq P_{R}}O \Big(\frac{1}{p} \Big) \ll  R-n.
\end{align}  
  
   It follows from \eqref{Bsumest} and \eqref{expest} that
\begin{align}
\begin{split}
\label{Bsumestsimplified}
\sum_{\mathcal{B}_{\mathcal{A},n}}  \prod_{1 \leq m\leq k} & |L(\tfrac{1}{2} + it_m, \chi)| |\mathcal{N}_{P_R}(\tfrac{1}{2} + it_m, \chi; \tfrac{1}{2}(a_m - 1))\mathcal{N}_{P_R}(\tfrac{1}{2} + it_m, \chi; \tfrac{1}{2}a_m)|^2  \\
\ll & \varphi(q) \exp \left(A/c_{n-1} - \log(1/c_n)/40c_n + O\left(R - n \right)\right) \prod_{p\leq P_{n-1}} \left(1 + \frac{1}{p} \sum_{1\leq j , m \leq k} \frac{a_j a_m}{p^{i(t_j - t_m)}}\right)\\
\ll &  \varphi(q) \exp \left(A/c_{n-1} - \log(1/c_n)/40c_n + O\left( R - n\right) \right) \prod_{p\leq P_R} \left(1 + \frac{1}{p} \sum_{1\leq j , m \leq k} \frac{a_j a_m}{p^{i(t_j - t_m)}}\right),
\end{split}
\end{align}
 where the last estimation above follows from using \eqref{expest} again. \newline
 
 The above discussions conclude our treatments for the case $n>1$. Now, for the case $n = 1$, we have the following result. 
\begin{prop}\label{prop:badSetInt}
Assuming the truth of GRH, for any $\mathcal{A} \subsetneq [m]$, and any constant $C>0$, we have
\begin{align}
\label{n1bound}
\sum_{\mathcal{B}_{\mathcal{A},1}} \prod_{1 \leq  m\leq k} L(\tfrac{1}{2} + it_m, \chi) \mathcal{N}_{P_R}(\tfrac{1}{2} + it_m, \chi; \tfrac{1}{2}(a_m - 1))^2 \overline{\mathcal{N}_{P_R}(\tfrac{1}{2} + it_m, \chi; \tfrac{1}{2}a_m)}^2 \ll_{C} \varphi(q)(\log q)^{-C}.
\end{align}
\end{prop}
\begin{proof}
Hölder's inequality yields that the left-hand side of \eqref{n1bound} is at most
\[
(\text{meas } \mathcal{B}_{\mathcal{A},1})^{1/3} \left(\sum_{\chi \in X_q} \prod_{1 \leq m\leq k} |L(\tfrac{1}{2} + it_m, \chi)|^3 \right)^{1/3}
\times \left(\sum_{\chi \in X_q} |\mathcal{N}_{P_R}(\tfrac{1}{2} + it_m, \chi; \tfrac{1}{2}(a_m - 1)) \mathcal{N}_{P_R}(\tfrac{1}{2} + it_m, \chi; \tfrac{1}{2}a_m) |^6    \right)^{1/3}.
\]
 It follows from the arguments in \cite{Sound2009} that the first sum above is $\ll \varphi(q) (\log q)^{O(1)}$.  Following the proof of Proposition \ref{prop:RHMVDP}, the second sum above is also $\ll \varphi(q) (\log q)^{O(1)}$. We further apply \eqref{powerest} and proceed as in the proof of \cite[lemma 2.7]{Curran} to see that $\text{meas } \mathcal{B}_{\mathcal{A},1} \ll \varphi(q) (\log q)^{-C'}$ for any constant $C'>0$. These estimations now imply the desired bound given in \eqref{n1bound}. This completes the proof of the proposition. 
\end{proof}

  We now deduce from \eqref{Bsumestsimplified} and Proposition \ref{prop:J1Lower} that for $n > 1$,
\begin{align*}
\sum_{\mathcal{B}_{\mathcal{A},n}}  \prod_{1 \leq m\leq k} & |L(\tfrac{1}{2} + it_m, \chi)| |\mathcal{N}_{P_R}(\tfrac{1}{2} + it_m, \chi; \tfrac{1}{2}(a_m - 1))\mathcal{N}_{P_R}(\tfrac{1}{2} + it_m, \chi; \tfrac{1}{2}a_m)|^2 \\
\ll & \exp\left(A/c_{n-1} - \log(1/c_n)/40c_n + O\left(R- n\right)\right) |\mathcal{J}_1|.
\end{align*}
 Upon summing over all $\mathcal{A} \subsetneq [m]$ and applying Proposition \ref{prop:badSetInt}, we see, via the definition of the $c_n$, that for a large constant $C>0$, 
\begin{align}
\begin{split}
\label{eq:J2est}
|\mathcal{J}_2| \ll & |\mathcal{J}_1| \sum_{2\leq n \leq R}\exp\left(A/c_{n-1} - \log(1/c_n)/40c_n + O\left( R - n\right)\right) + O_C\left(\varphi(q)(\log q)^{-C}\right) \\
\ll & |\mathcal{J}_1| \sum_{2\leq n \leq R}\exp\left(e^{-n} (\log\log q)^2 (O(1) + \tfrac{1}{40} n - \tfrac{1}{20} \log\log\log q)\right) + O(R - n) ) + O_C\left(\varphi(q)(\log q)^{-C}\right) \\
\ll & |\mathcal{J}_1| \sum_{2\leq n \leq R}\exp\left(e^{-n} (\log\log q)^2 (O(1) + \tfrac{1}{40} n - \tfrac{1}{20} \log\log\log q)\right) + O(R - n) ),
\end{split}
\end{align}
 where the last estimation above follows from the observation that $|\mathcal{J}_1| \gg \varphi(q)$. \newline

  Note that the condition $P_R \leq q^{\delta}$ implies that 
\begin{align}
\begin{split}
\label{Rbound}
R \leq 2 \log\log\log q + \log \delta.
\end{split}
\end{align}
 It follows from this and \eqref{eq:J2est} that
\begin{align*}
|\mathcal{J}_2| \ll  |\mathcal{J}_1|\sum_{2\leq n \leq R} \exp\left(e^{-n} (\log\log q)^2 (O(1) + \tfrac{1}{40}\log \delta)\right) + O(R - n) ).
\end{align*}

  We then make the change of variable $R-n \to j$, take $\delta$ small enough (so that the term $O(1) + \tfrac{1}{40}\log \delta$ above is negative) and apply the crude bound $R \leq 2 \log\log\log q$ from \eqref{Rbound}.  This leads to
\begin{align}
\begin{split}
\label{J2bound10}
|\mathcal{J}_2| \ll &  |\mathcal{J}_1| \sum_{j \geq 0} \exp\left(e^j (O(1) + \tfrac{1}{40}\log \delta) + O(j) \right) 
\ll   |\mathcal{J}_1| \sum_{j \geq 0} \exp\left( (O(1) + \tfrac{1}{40}\log \delta)j + O(j) \right) \\
\ll & |\mathcal{J}_1|\sum_{j \geq 1} \exp\left((O(1) + \tfrac{1}{40}\log \delta ) j \right).
\end{split}
\end{align}
Again, taking $\delta$ sufficiently small to ensure the term $O(1) + \tfrac{1}{40}\log \delta$ above is negative, we deduce from \eqref{J2bound10} that
\begin{align*}
\begin{split}
|\mathcal{J}_2| \leq &  |\mathcal{J}_1|\exp\left(O(1) + \tfrac{1}{40}\log \delta  \right).
\end{split}
\end{align*}
If needed, we further reduce $\delta$ in terms of $\bf{a}$ to make the implicit constant above not exceed $1/2$.  This now completes the proof of Proposition \ref{prop:J2J1comp}.

\subsection{Proof of Proposition \ref{prop:RHtwisted}}

We apply Lemma \ref{lem:logZetaUpperBound} by setting $k=a_1=1, x = P_{R}$ there to see that for any $\chi \in X^*_q$ and $1 \leq m \leq k$, 
\begin{align*}
|L(\tfrac{1}{2} + it_m, \chi)| 
\ll &
\exp\Big(\Re \Big( \sum_{1 \leq j < n}  \mathcal{P}_{j,P_{R}} (\tfrac{1}{2}  + it_m, \chi) \Big) + \frac {A+1}{ c_{R}} \Big) \ll \exp\Big(\Re \Big( \sum_{1 \leq j < n}  \mathcal{P}_{j,P_{R}} (\tfrac{1}{2}  + it_m, \chi) \Big) \Big).
\end{align*}  

The above, together with the definition of $\mathcal{T}_k$ and Lemma \ref{lem:expTaylorSeries}, gives
\begin{align*}
\sum_{\chi \in \mathcal{T}_k}  \prod_{1 \leq m\leq k} & |L(\tfrac{1}{2} + it_m, \chi)| |\mathcal{N}_{P_R}(\tfrac{1}{2} + it_m, \chi; \tfrac{1}{2}(a_m - 1))\mathcal{N}_{P_R}(\tfrac{1}{2} + it_m, \chi; \tfrac{1}{2}a_m)|^2   \\
\ll &   \sum_{\chi \in \mathcal{T}_k}\prod_{m = 1}^k \prod_{j \leq R} \exp\left(\Re \mathcal{P}_{j,P_R} (\tfrac{1}{2}  + it_m, \chi) \right) \times |\mathcal{N}_{P_R}(\tfrac{1}{2} + it_m, \chi; \tfrac{1}{2}(a_m - 1))\mathcal{N}_{P_R}(\tfrac{1}{2} + it_m, \chi; \tfrac{1}{2}a_m)|^2  \\
\ll &
\sum_{\chi \in X_q} \prod_{m = 1}^k  |\mathcal{N}_{P_R} (\tfrac{1}{2}+ it_m, \chi ; \tfrac{1}{2}) |^2 \times |\mathcal{N}_{P_R}(\tfrac{1}{2} + it_m, \chi; \tfrac{1}{2}(a_m - 1))\mathcal{N}_{P_R}(\tfrac{1}{2} + it_m, \chi; \tfrac{1}{2}a_m)|^2  .
\end{align*}

Applying Proposition \ref{prop:RHMVDP} with $n = R + 1$ there and \eqref{orthorelprod} gives that the last expression above is
\[
\ll \varphi(q) \prod_{p \leq P_R} \left(1 + \sum_{1\leq j,m \leq k} \frac{a_j a_m}{p^{1+i(t_j - t_m)}}\right).
\]
This now completes the proof of Proposition \ref{prop:RHtwisted}.

\subsection{Conclusion}

 We deduce from lemma \ref{lem:incrementBound}, Propositions \ref{prop:I0Llower}-\ref{prop:RHtwisted} that by the definition of $K_J$,
\begin{align*}
\begin{split}
\sum_{J\leq R} \big|\mathcal{S}_0^{(J)} - \mathcal{S}_0^{(J-1)} \big| \ll & |S_0^{(R)}| \sum_{J\leq R} e^{-50a_\ast^2K_J}=|I_0^{(R)}| \sum_{J\leq R}\exp\left(-50a_\ast^2 \frac{(\log\log q)^{3/2}}{e^{3J/4}}\right) \\
\ll & |S_0^{(R)}| \sum_{J\leq 2 \log\log\log q + \log \delta}\exp\left(-50a_\ast^2 \frac{(\log\log q)^{3/2}}{e^{3J/4}}\right),
\end{split}
\end{align*}
 where the last line above follows from \eqref{Rbound}.  Changing of variable by setting $j=2 \log\log\log q + \log \delta-J$ gives that the last sum above is bounded above by
\[
\sum_{j\geq 1} \exp\left(-50 a_\ast^2 e^{-3\log\delta / 4}  e^{3j/4}\right) \leq \sum_{j\geq 1} \exp\left(-50a_\ast^2 e^{-3\log\delta / 4}    j \right) \ll \exp\left(-50a_\ast^2 e^{-3\log\delta / 4} \right) .
\]
We now choose $\delta$ sufficiently small in terms of ${\bf a}$, arriving at
\[
\sum_{J\leq R} \big|\mathcal{S}_0^{(J)} - \mathcal{S}_0^{(J-1)} \big| \leq \frac{|S_0^{(R)}|}{2}.
\]
 It now follows from \eqref{S0decomp}, Proposition \ref{prop:I0Llower} and the above that we have
\begin{align}
\begin{split}
\label{S0lowerbound}
 |S_0| \gg & \varphi(q) \prod_{p \leq P_R} \left(1 + \sum_{1\leq j,m \leq k} \frac{a_j a_m}{p^{1+i(t_j - t_m)}}\right)=\varphi(q) \prod_{p \leq P_R} \left(1 + \frac{|\sum_{m = 1}^k a_m p^{-it_m}|^2}{p} \right) \\
\gg & \varphi(q) \exp \left(\sum_{p \leq P_R} \frac{|\sum_{m = 1}^k a_m p^{-it_m}|^2}{p} \right), 
\end{split}
\end{align}
 where the last bound above follows by noting from \eqref{1xbound} that for $p >(10+a)^2=:K$ (recall that $a=\sum^k_{m=1}a_m$),
\begin{align*}
\begin{split}
\exp &\left( -\sum_{K<p \leq P_R} \frac{|\sum_{m = 1}^k a_m p^{-it_m}|^2}{p} \right)\prod_{K< p \leq P_R} \left(1 + \frac{|\sum_{m = 1}^k a_m p^{-it_m}|^2}{p} \right) \\
& \hspace*{1cm} \gg \prod_{K< p \leq P_R} \left(1-\frac{|\sum_{m = 1}^k a_m p^{-it_m}|^2}{p} \right) \left(1 + \frac{|\sum_{m = 1}^k a_m p^{-it_m}|^2}{p} \right) 
\gg \prod_{p} \left(1 -\frac{|\sum_{m = 1}^k a_m p^{-it_m}|^4}{p^2}\right)  \gg  1. 
\end{split}
\end{align*} 
 
   We further deduce from \eqref{S0lowerbound} that
\begin{align}
\begin{split}
\label{S0lowerboundzeta}
 |S_0| 
\gg & \varphi(q) \exp \left(\sum_{p \leq q} \frac{|\sum_{m = 1}^k a_m p^{-it_m}|^2}{p} \right) \gg    \varphi(q) (\log q)^{a_1^2 + \cdots + a_m^2} 
\prod_{1\leq j < l \leq k} |\zeta(1 + i(t_j - t_l) + 1/ \log q )|^{2a_j a_l}.
\end{split}
\end{align}
The last estimation above follows from \eqref{merten} and \eqref{mertenstype}, while the penultimate follows from the observation that by \eqref{merten}, 
\begin{align*}
\begin{split}
\sum_{P_R < p \leq q} \frac{|\sum_{m = 1}^k a_m p^{-it_m}|^2}{p}  \ll \sum_{P_R < p \leq q}O \Big( \frac 1{p} \Big) \ll 1. 
\end{split}
\end{align*} 

  We see readily from \eqref{S0lowerboundzeta} that the assertion of Proposition \ref{prop:I0Bound} holds. This completes its proof.

\section{Proof of Proposition \ref{prop:IkBound}}
 Recall that
\begin{align*}
\mathcal{S}_m = \sum_{\chi \in \mathcal{T}_m} |L(\tfrac{1}{2} + i t_m, \chi)|^{k} \prod_{1 \leq  j\leq R} &\exp\left(2(a_m-k)\Re \mathcal{P}_{j,P_R}(\tfrac{1}{2}+it_k, \chi) \right)  \prod_{\substack{1 \leq \ell \leq k \\ \ell \neq m}} \prod_{1 \leq j\leq R} \exp\left(2a_\ell\Re \mathcal{P}_{j,P_R}(\tfrac{1}{2}+ia_\ell, \chi)) \right) .
\end{align*}
 Similar to the proof of Proposition \ref{prop:RHtwisted}, we apply Lemma \ref{lem:logZetaUpperBound} by setting $k=a_1=1$, $x = P_{R}$ there together with Lemma \ref{lem:expTaylorSeries} to see that
\begin{align*}
\mathcal{S}_m \ll \sum_{\chi \in X_q} |\mathcal{N}
_{P_R}(\tfrac{1}{2} &+ it_m, \chi;k) \mathcal{N}_{P_R}(\tfrac{1}{2} + it_m, \chi;2(a_m - k))|^2 \times \prod_{\substack{1 \leq \ell \leq k \\ \ell \neq m}}  |\mathcal{N}_{P_R}(\tfrac{1}{2} + it_\ell, \chi;a_\ell)|^2  .
\end{align*}
 We further recall the definitions made preceding Proposition \ref{prop:Njcoeffs}.  For $j > 1$, the coefficients $a_j(n)$ of the Dirichlet polynomial
\[
\mathcal{N}_{j,P_R}(\tfrac{1}{2} + it_m, \chi;k) \mathcal{N}_{j,P_R}(\tfrac{1}{2} + it_m, \chi;2(a_m - k)) \prod_{\substack{1 \leq \ell \leq m \\ \ell \neq k}}  \mathcal{N}_{j,P_R}(\tfrac{1}{2} + it_\ell, \chi;a_\ell) =:\sum_n \frac{a_j(n)}{n^{1/2}}
\]
equals the $m + 1$ fold Dirichlet convolution of $g_{P_R}(n; k) n^{-i t_m} c_j(n), g_{P_R}(n; a_m -k ) n^{-i t_m} c_j(n)$ and $g_{P_R}(n; a_\ell) n^{-i t_\ell} c_j(n)$ for all $\ell \leq k, \ell \neq m$.
If $j = 1$, a similar relation holds with $f_{P_R}$ in place of $g_{P_R}$.
As before, one may replace $c_j(n)$ with ${\bf 1}_{p|n \Rightarrow p\in (P_{j-1},P_j]}$, incurring an error of size $O(e^{-50a_\ast^2K_j})$, to obtain the multiplicative coefficients $a_j'(n)$ such that for $P_{j-1} < p \leq P_j$, 
\begin{align*}
\begin{split}
|a'_j(p)|^2 =  \sum_{1\leq l,k \leq m} \frac{a_l a_k}{p^{i(t_l - t_k)}} + O\left(\frac{\log p}{\log P_R}\right), \quad \sum_{r \geq 2} \frac{|a'_j(p^r)|^2}{p^r}=O \Big( \frac{1}{p^2} \Big). 
\end{split}
\end{align*}
As in the proof of Proposition \ref{prop:RHtwisted}, we deduce from the above that
\[
\mathcal{S}_m \ll \varphi(q) \prod_{p\leq P_R} \left(1 + \frac{1}{p} \sum_{1\leq j,m \leq k} \frac{a_j a_k}{p^{i(t_j - t_m)}} \right) \ll   \varphi(q) (\log q)^{a_1^2 + \cdots + a_m^2} \prod_{1\leq j < l \leq k} |\zeta(1 + i(t_j - t_l) + 1/ \log q )|^{2a_j a_l},
\]
where the final bound above emerges using arguments similar to those in \eqref{Tupperbound}.  This completes the proof of Proposition \ref{prop:IkBound}.

\vspace*{.5cm}

\noindent{\bf Acknowledgments.}  P. G. is supported in part by NSFC grant 12471003 and L. Z. by the FRG Grant PS71536 at the University of New South Wales.

\bibliography{biblio}
\bibliographystyle{amsxport}

\end{document}